\documentclass{article}

\usepackage{arxiv}

\usepackage[utf8]{inputenc} 
\usepackage[T1]{fontenc}    
\usepackage{hyperref}       
\usepackage{url}            
\usepackage{booktabs}       
\usepackage{amsfonts}       
\usepackage{nicefrac}       
\usepackage{microtype}      
\usepackage{lipsum}		
\usepackage{graphicx}
\usepackage{natbib}
\usepackage{doi}
\usepackage{amsmath}
\usepackage{amsthm}

\newtheorem{proposition}{Proposition}
\newcommand{\notimplies}{%
  \mathrel{{\ooalign{\hidewidth$\not\phantom{=}$\hidewidth\cr$\implies$}}}}
\usepackage{algorithm}
\usepackage{algorithmic}
\usepackage{enumitem}
\graphicspath{ {./figures/} }

\title{Learning Control-Affine Reduced-Order Models via Autoencoders}

\author{ \href{https://orcid.org/0000-0002-7940-8820}{\includegraphics[scale=0.06]{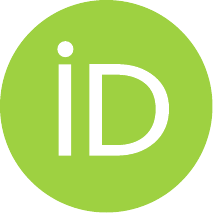}\hspace{1mm}Ali Mjalled} \\
	Automatic Control and Systems Theory\\
	Ruhr-Universit\"at Bochum\\
	Universit\"atstraße 150, 44801, Bochum \\
	\texttt{ali.mjalled@ruhr-uni-bochum.de} \\
	\And
	\href{https://orcid.org/0000-0002-9277-2061}{\includegraphics[scale=0.06]{orcid.pdf}\hspace{1mm}Martin M\"onnigmann} \\
	Automatic Control and Systems Theory\\
	Ruhr-Universit\"at Bochum\\
	Universit\"atstraße 150, 44801, Bochum \\
	\texttt{martin.moennigmann@ruhr-uni-bochum.de} \\
}

\renewcommand{\shorttitle}{\textit{arXiv} Template}

\hypersetup{
pdftitle={A template for the arxiv style},
pdfsubject={q-bio.NC, q-bio.QM},
pdfauthor={David S.~Hippocampus, Elias D.~Striatum},
pdfkeywords={First keyword, Second keyword, More},
}

\begin{document}
\maketitle

\begin{abstract}
    We present in this paper a framework for the identification of control-affine reduced-order models (ROMs).
    The proposed method utilizes autoencoders (AEs) to transform the high-dimensional states, and potentially the high-dimensional inputs, into reduced latent ones suitable for control-affine state-space dynamics.
    This is achieved by simultaneous training of the AE and the state-space model.
    In addition, we extend the discrete ROM formulation to a sequence-based model, which processes state and input histories to improve prediction accuracy while preserving the control-affine structure.
    We motivate our framework by applying feedback linearization to the derived models, and we present guidelines for its efficient use.
    The proposed framework is assessed on two numerical examples and its performance is compared to a baseline model, where the AE identifies a latent space with linear state-space dynamics.
    The assessment involves evaluating the prediction accuracy of the ROM on test data and its effectiveness in controlling the system to a desired state or trajectory.
\end{abstract}


\section{Introduction}
The simulations of engineering systems governed by complex nonlinear PDEs are often computationally expensive.
High-fidelity models, derived from fine discretizations, are typically too expensive for applications like real-time optimization or control.
This challenge has motivated the development of reduced-order models (ROMs) over the past decades, which aim to capture essential dynamics while reducing the computational cost.

Traditional reduced-order modeling techniques, such as the proper orthogonal decomposition (POD) followed by Galerkin projection, rely on identifying a low-dimensional subspace that captures the dominant patterns and projecting the governing equations onto it to obtain a reduced set of ODEs, representing the ROM (see, e.g., \cite{benner2015survey} for a survey on projection-based methods).
However, a fundamental limitation of these methods is their reliance on explicit knowledge of the governing equations.
In many practical scenarios, only simulation data is available and the underlying equations are either unknown or too complex to manipulate.
Consequently, the task of model reduction becomes data-driven and aims to construct dynamical models purely from observations or simulated data.

Advances in machine learning have rapidly accelerated the development of data-driven reduced-order modeling.
These techniques often address key limitations of standard approaches.
For example, autoencoders (AEs) can identify nonlinear manifolds, providing a more expressive alternative to the linear subspaces produced by POD \citep{lee2020model}.
Similarly, numerous neural network architectures have been employed to model the intrinsic dynamics, including recurrent neural networks \citep{maulik2021reduced, hasegawa2020cnn}, temporal convolutional networks \citep{xia2023hierarchical, wu2020data}, neural ODEs \citep{mjalled2024improved, linot2022data}, and transformers \citep{solera2024beta, hemmasian2023reduced}.
These methods allow us to model the nonlinear dynamics within the reduced space.

System identification is closely related to data-driven reduced-order modeling, especially for systems with external inputs.
While system identification traditionally focuses on learning state-space models from input-output data \citep{masti2021learning}, the two fields converge when the output data is high-dimensional (e.g., camera frames). 
In such cases, learning a model directly in the high-dimensional space is computationally expensive, making a low-dimensional latent representation essential.
Consequently, the problem becomes one of jointly learning an intrinsic coordinate transformation and a state-space model within the resulting latent space.
Established methods like the dynamic mode decomposition identify linear latent models from high-dimensional snapshots \citep{proctor2016dynamic}.
More generally, nonlinear state-space models can also be learned.
For instance, \cite{wahlstrom2015learning} combined an AE with a NARX model to learn deep dynamical models from images, while \cite{beintema2021non} developed a framework to learn state-space models directly from video and input signals using a deep encoder.

While ROMs enable critical engineering applications such as design optimization \citep{park2013reduced}, state estimation \citep{sommer2023estimating}, control \citep{tomasetto2025latent}, and digital twin development \citep{kapteyn2022data}, their practical utility can be hindered by overly complex latent dynamics.
Although significant effort has been dedicated to discover interpretable latent representations using variational \citep{higgins2017beta} or deterministic \citep{schwarz2025disentangled} frameworks, the resulting temporal evolution can remain too complex.
When modeled with complex architectures, these dynamics yield models that are too computationally expensive for real-time use, such as in the model predictive control, where an optimization task must converge within a strict sampling time.
Furthermore, such dynamical models often lack interpretability.
To alleviate these challenges, recent research has focused on constraining the latent dynamics to specific structures.
A prominent approach enforces linear dynamics, thereby enabling the direct application of linear systems theory.
This objective is deeply connected to Koopman operator theory, which seeks eigenfunctions that provide intrinsic coordinates for a globally linear representation of nonlinear dynamics.
In this context, \cite{lusch2018deep} introduced a framework in which a modified AE discovers approximations of Koopman eigenfunctions from data, with extensions to handle systems with continuous spectra using an auxiliary network for frequency parameterization.
A concurrent approach by \cite{otto2019linearly} employs a linearly-recurrent AE architecture to learn low-dimensional Koopman-invariant subspaces, incorporating balanced model reduction and nonlinear reconstruction techniques.
Beyond purely autonomous dynamics, other works have enforced linearity for control purposes.
For example, \cite{ishize2024flow} imposed a linear latent structure with inputs to facilitate linear control design for flow control applications, while the earlier Embed to Control framework by \cite{watter2015embed} uses a variational AE framework that enforces a locally linear latent transition, enabling classical control design in the latent space.

Apart from imposing linearity, another approach replaces neural network dynamics with interpretable, closed-form models identified using the sparse identification of nonlinear dynamics (SINDy) method \citep{brunton2016discovering}, which uses sparse regression to identify a set of ODEs from a library of candidate functions.
This principle was first integrated with deep learning by \cite{champion2019data}, who used an AE to discover a latent space where dynamics admit a sparse representation.
Subsequent work by \cite{bakarji2023discovering} has extended on this idea and applied a similar framework to learn coordinate transformations from delay embedded spaces.
Likewise, \cite{mars2024bayesian} introduced a Bayesian SINDy AE to promote sparsity in the latent space and to provide uncertainty estimates for the identified dynamics.
More recently, \cite{gao2025sparse} proposed SINDy-SHRED, which is a recurrent AE architecture with SINDy regularization.
We note that SINDy has also been extend to model systems with inputs and demonstrated successful results to control low dimensional models \citep{brunton2025machine}.

While linear or sparse closed-form latent dynamics facilitate the modeling and control of high-dimensional systems, they pose their own limitations.
The assumption of linearity may be insufficient to capture complex nonlinear dynamics, especially in model reduction.
Koopman theory addresses this by lifting the state to a higher-dimensional space where dynamics are linear, but this contradicts the objective of reduction, where we do not aim to further increase the dimension of our states.
Conversely, while sparse models offer greater generality and interpretability, they are inherently nonlinear.
This necessitates complex nonlinear controllers, and their accuracy is highly sensitive to the choice of candidate functions in the SINDy library.
To overcome these challenges, we propose a novel ROM that enforces a nonlinear control-affine state-space structure on the latent dynamics, a formulation that, to the best of our knowledge, has not been previously explored for data-driven ROMs.
This approach offers a powerful compromise.
First, it retains the capacity for accurate nonlinear modeling while enabling the use of linear control techniques through feedback linearization \citep{isidori1995nonlinear}, after a suitable coordinate transformation performed using the AE.
Furthermore, the control-affine structure naturally aligns with a wide class of engineering systems where the dynamics are nonlinear in the state but affine in the control input. 

For control-affine systems, \cite{peitz2020data} proved that the latent dynamics is bilinear in the Koopman invariant subspace.
Building on this, \cite{chakrabarti2025temporally} proposed a bilinear recurrent AE by mapping to a Koopman-invariant subspace where the dynamics are bilinear.
In this work, we take a different path as we do not aim to learn an approximate Koopman-invariant subspace to obtain bilinear dynamics.
Instead, we directly learn an intrinsic coordinate in which the dynamics is nonlinear control-affine.

This paper is organized as follows.
Section \ref{sec:modeling} provides a formal problem definition and outlines the standard approach to build data-driven ROMs.
Section \ref{sec:control-affine} introduces the proposed control-affine ROM, while Sec.~\ref{sec:seq-based-affine} extends it to sequence-based formulation to improve prediction accuracy while preserving its structure.
Section \ref{sec:FL} presents feedback linearization guidelines for the proposed ROM.
Numerical examples are presented in Sec.~\ref{sec:examples}, and conclusions with outlooks are discussed in Sec.~\ref{sec:conclusions}.

\section{Data-driven reduced-order modeling}
\label{sec:modeling}
\subsection{Problem formulation}
We consider in this work discrete-time dynamical systems
\begin{equation}
\label{eq:high-dim-dyn}
    \mathbf{x}_{k+1} = \mathbf{F}(\mathbf{x}_k,\mathbf{u}_k),
\end{equation}
where $\mathbf{x}_k \in \mathcal{X}$ and $\mathbf{u}_k\in \mathbb{R}^m$ denote the state and input vectors of the system, respectively, at time $t_k=k\Delta t$, with fixed sampling interval $\Delta t$.
Here, $\mathcal{X}$ denotes a finite-dimensional Hilbert space equipped with an inner product $\langle\cdot,\cdot\rangle_{\mathcal{X}}$ and induced norm $\|\cdot\|_{\mathcal{X}}$.
This setting encompasses both vector-valued states ($\mathcal{X} = \mathbb{R}^n$ with the Euclidean inner product) and structured states such as spatial fields (e.g., $\mathcal{X} = \mathbb{R}^{N_X \times N_Y}$ with the Frobenius inner product).
The dynamics map $\mathbf{F}: \mathcal{X} \times \mathbb{R}^m \rightarrow \mathcal{X}$ is in general nonlinear and unknown; only discrete measurements of the states and inputs are available $\{\mathbf{x}_k, \mathbf{u}_k\}_{k=0}^{N-1}$, from which we seek an approximation $\hat{\mathbf{F}}$.

In high-dimensional settings ($\mathrm{dim}(\mathcal{X})=n \gg 1$), direct identification of $\mathbf{F}$ from data poses significant computational challenges due to the complexity and scale of the system.
Fortunately, many physical systems evolve on a low-dimensional manifold $\mathcal{M}\subset \mathcal{X}$, such that the dynamics is governed by a reduced number of degrees of freedom \citep{holmes2012turbulence}.
Specifically, we assume that $\mathcal{M}$ has intrinsic dimension $r \ll n$, and that the high-dimensional state $\mathbf{x}_k$ can be encoded by a latent vector $\mathbf{z}_k \in \mathbb{R}^r$ capturing its essential features, such that $\mathbf{z}_k = \mathcal{E}(\mathbf{x}_k)$, where $\mathcal{E}:\mathcal{X} \rightarrow \mathbb{R}^r$ is a dimensionality reduction map commonly referred to as the encoder.
Accordingly, the dynamics within the latent space are governed by a (generally unknown) reduced-order map $\mathbf{G}:\mathbb{R}^r \times \mathbb{R}^m \rightarrow\mathbb{R}^r$, i.e.,
\begin{equation}
    \label{eq:low-dim-dyn}
    \mathbf{z}_{k+1} = \mathbf{G}(\mathbf{z}_k,\mathbf{u}_k).
\end{equation}

To reconstruct the high-dimensional state from its latent representation, the inverse of $\mathcal{E}$ is required.
However, the mapping $\mathcal{E}$ cannot be injective and therefore does not admit a global inverse.
Instead, an approximate inverse mapping, commonly referred to as the decoder $\mathcal{D}$, is constructed.
Therefore, the retrieval of high-dimensional information from its low-dimensional counterpart, which is essential for reduced order modeling, is usually subject to an approximation error $\mathbf{e}_k$, expressed as:
\begin{equation}
    \label{eq:error}
    \mathbf{e}_k = \mathbf{x}_k - \Tilde{\mathbf{x}}_k = \mathbf{x}_k - \mathcal{D}(\mathcal{E}(\mathbf{x}_k)).
\end{equation}

Accordingly, the task of learning $\mathbf{F}$ from data is reformulated as identification of the encoder $\mathcal{E}$, the latent dynamics model $\mathbf{G}$, and the decoder $\mathcal{D}$.
In the following, we refer to the combination of $\mathcal{E}$, $\mathbf{G}$ and $\mathcal{D}$ as the ROM.  

\subsection{Decoupled neural network parameterization}
It is common practice to model the components $\mathcal{E}$, $\mathbf{G}$ and $\mathcal{D}$ using neural networks due to their ability to approximate nonlinear functions.
In the following, we denote these functions as $\mathcal{E}_\theta$, $\mathbf{G}_\varphi$ and $\mathcal{D}_\phi$, where $\theta, \varphi$ and $\phi$ represent the learnable parameters of each network, respectively.
When no prior structure is imposed on the latent dynamics, the training procedure is composed of two main decoupled steps.

\subsubsection{Autoencoder training}
Within this framework, the encoder $\mathcal{E}_\theta$ and decoder $\mathcal{D}_\phi$ are trained simultaneously using an AE structure, which is a special type of neural networks trained to reproduce its inputs \citep{rumelhart1985learning}.
The training parameters $\theta$ and $\phi$ are optimized by minimizing the reconstruction error given in \eqref{eq:error}.
Specifically, the following loss function is minimized:
\begin{equation}
    \label{eq:opt-ae-alone}
    \mathcal{L}_{\text{AE}} = \frac{1}{N} \sum_{k=0}^{N-1} \| \mathbf{x}_k - \mathcal{D}_\phi(\mathcal{E}_\theta(\mathbf{x}_k)) \|_{\mathcal{X}}^2,
\end{equation}
where $\|\cdot\|_{\mathcal{X}}$ denotes the norm induced by the inner product of $\mathcal{X}$.
Apart from the detailed neural network architectures and the hyperparameters of the optimizer algorithm, the main modeling hyperparameter at this stage is the latent dimension $r$.
This should be selected to ensure that the encoder $\mathcal{E}_\theta$ produces a compact latent representation that retains the essential information necessary for the reconstruction.
Depending on the structure of $\mathcal{X}$, the mappings $\mathcal{E}_\theta$ and $ \mathcal{D}_\phi$ may be parameterized using fully connected or convolutional neural networks

\subsubsection{Learning latent dynamics}
Once the AE is trained, the encoder $\mathcal{E}_{\theta}$ is used to transform all high-dimensional states $\{\mathbf{x}_k\}_{k=0}^{N-1}$ into their corresponding latent representations $\{\mathbf{z}_k=\mathcal{E}_{\theta}(\mathbf{x}_k)\}_{k=0}^{N-1}$.
The latent dynamics model $\mathbf{G}_\varphi$ is then trained to fit the evolution of the latent variables using an $M$-step rollout loss.
To evaluate the model performance over this horizon, we denote $\hat{\mathbf{z}}_{l|k}$ as the prediction at step $l$ ahead, starting from time $k$.
The loss function for optimizing the latent dynamics reads
\begin{equation}
\label{eq:opt-lat-dyn-alone}
\mathcal{L}_{\text{lat}} = \frac{1}{N-M} \sum_{k=0}^{N-M-1} \sum_{l=1}^M \| \mathbf{z}_{k+l} - \hat{\mathbf{z}}_{l|k} \|_2^2,
\end{equation}
where $\| \cdot \|_2$ denoted the L2 norm of a vector, and the predicted sequence is computed recursively starting from the encoded state $\mathbf{z}_k$:
\begin{equation*}
\hat{\mathbf{z}}_{0|k} := \mathbf{z}_k, \quad \hat{\mathbf{z}}_{l+1|k} := \mathbf{G}_\varphi(\hat{\mathbf{z}}_{l|k}, \mathbf{u}_{k+l}), \quad l=0, \dots, M-1.
\end{equation*}
In addition to the rollout horizon $M$ and the hyperparameters of the optimizer, the architecture and the structure of $\mathbf{G}_\varphi$ play a crucial role in determining the performance of the model.
A variety of neural architectures have been employed to model the nonlinear latent dynamics, including LSTM networks \citep{mjalled2023reduced}, temporal convolutional networks \citep{xia2023hierarchical}, and transformers \citep{hemmasian2023reduced}, to name just a few.
Alternatively, instead of allowing $\mathbf{G}_\varphi$ to represent an arbitrary nonlinear function, it is often beneficial to restrict it to a specific family of functions.
This reduces the search space during optimization and can lead to models that are more interpretable and generalizable.

\section{Control-affine latent dynamics}
\label{sec:control-affine}
In this work, we model the latent dynamics as a nonlinear control-affine discrete system in the form
\begin{equation}
    \label{eq:latent-affine}
    \mathbf{z}_{k+1} = \mathbf{a}_{\varphi_1}(\mathbf{z}_k) + \mathbf{B}_{\varphi_2}(\mathbf{z}_k)\cdot \mathbf{u}_k,
\end{equation}
where $\mathbf{a}_{\varphi_1}:\mathbb{R}^r \rightarrow \mathbb{R}^r$ is the drift network, parameterized with $\varphi_1$, and $\mathbf{B}_{\varphi_2}:\mathbb{R}^r \rightarrow\mathbb{R}^{r \times m}$ is the input network, parameterized with $\varphi_2$.
This structure decomposes the latent dynamics into an autonomous (drift) part, which governs the temporal evolution in the absence of control inputs, and a control-influenced part, which interacts with the control signal $\mathbf{u}_k$ in an affine manner.
Imposing the control-affine structure for the latent dynamics is advantageous for many reasons.
First, it introduces an inductive bias that aligns with many real world physical systems, e.g., robot manipulators, aerial and ground vehicles, chemical reactors, etc. 
Second, it enables the use of well-established analysis and control methods tailored for this family of models, specifically, feedback linearization (Sec.~\ref{sec:FL}).
Finally, this structure facilitates interpretability and modularity, allowing to analyze or improve each part independently.

The decoupled training approach introduced earlier for general nonlinear latent dynamics does not explicitly account for structural constraints on the latent model.
In particular, when the latent dynamics are restricted to a control-affine form as in \eqref{eq:latent-affine}, an autoencoder trained independently may emphasize reconstruction of features that are not well aligned with this structure.
As a result, the learned latent representation may not be ideally suited for approximating the dynamics under the imposed control-affine assumption.
To address this, we adopt a joint training strategy in which the encoder, decoder, and latent dynamics model are optimized simultaneously.
This setup promotes the identification of a latent representation that both reconstructs the data and is compatible with the chosen control-affine structure.
The multi-objective loss function used for this purpose reads
\begin{equation}
    \label{eq:mult-obj-loss}
    \mathcal{L} = \gamma_1 \mathcal{L}_{\text{AE}} + \gamma_2 \mathcal{L}_{\text{lat-cons}} + \gamma_3\mathcal{L}_{\text{dyn}},
\end{equation}
where:
\begin{itemize}
    \item $\mathcal{L}_{\text{AE}}$ is the autoencoder loss given in \eqref{eq:opt-ae-alone},
\item $\mathcal{L}_{\text{lat-cons}}$ is the latent consistency loss obtained from \eqref{eq:opt-lat-dyn-alone} by substituting $\mathbf{z}_{k+l} = \mathcal{E}_\theta(\mathbf{x}_{k+l})$, i.e.,
\begin{equation}
\label{eq:loss-lat-dyn}
\mathcal{L}_{\text{lat-cons}} = \frac{1}{N-M} \sum_{k=0}^{N-M-1} \sum_{l=1}^M \| \mathcal{E}_\theta(\mathbf{x}_{k+l}) - \hat{\mathbf{z}}_{l|k} \|_2^2,
\end{equation}
\item $\mathcal{L}_{\text{dyn}}$ is the end-to-end loss coupling all components together in the high-dimensional space, i.e.,
\begin{equation}
\label{eq:loss-dyn}
\mathcal{L}_{\text{dyn}} = \frac{1}{N-M} \sum_{k=0}^{N-M-1} \sum_{l=1}^{M} \left\| \mathbf{x}_{k+l} - \mathcal{D}_\phi(\hat{\mathbf{z}}_{l|k}) \right\|_{\mathcal{X}}^2.
\end{equation}
\end{itemize}
All rollout predictions $\hat{\mathbf{z}}_{l|k}$ are computed recursively using the control-affine model:
\begin{equation}
\label{eq:rollouts}
\hat{\mathbf{z}}_{0|k} := \mathcal{E}_\theta(\mathbf{x}_k), \quad
\hat{\mathbf{z}}_{l+1|k} := \mathbf{a}_{\varphi_1}(\hat{\mathbf{z}}_{l|k}) + \mathbf{B}_{\varphi_2}(\hat{\mathbf{z}}_{l|k})\cdot \mathbf{u}_{k+l}, \quad l = 0,\dots,M-1.
\end{equation}
The scalars $\gamma_1$, $\gamma_2$, and $\gamma_3$ are weighting coefficients controlling the trade-off between reconstruction accuracy, latent consistency and overall performance.
While it is possible to obtain an accurate ROM by minimizing a loss composed of $\mathcal{L}_{\text{AE}}$ and $\mathcal{L}_{\text{dyn}}$, the inclusion of $\mathcal{L}_{\text{lat-cons}}$ term is essential for control applications.
Without this term, decoupling the AE and the dynamical model is not possible as there is not guarantee that the learned dynamics within the latent space is consistent with the ground-truth encoded dynamics.
In such cases, even if the high-dimensional predictions appear accurate, this might be due to the decoder compensating for latent representations that do not follow the true dynamics.
Therefore, minimizing $\mathcal{L}_{\text{lat-cons}}$ not only enforces consistency, but enables analysis and control directly in the reduced space.
We note that minimizing $\mathcal{L}_{\text{lat-cons}}$ helps reduce $\mathcal{L}_{\text{dyn}}$, but the opposite is not guaranteed.
This relationship is formalized in the following proposition.

\begin{proposition}
    \label{prop:loss}
Let \(\mathcal{L}_{\mathrm{lat\text{-}cons}}\) be the latent dynamics loss defined in \eqref{eq:loss-lat-dyn} and \(\mathcal{L}_{\mathrm{dyn}}\) the end-to-end ROM loss defined in \eqref{eq:loss-dyn}.  Assume that the autoencoder achieves perfect reconstruction on the training set, i.e.,
\[
\mathcal{D}_\phi\bigl(\mathcal{E}_\theta(\mathbf{x})\bigr)=\mathbf{x}
\quad\forall\,\mathbf{x}\text{ in the training set}.
\]
Then
\[
\mathcal{L}_{\mathrm{lat\text{-}cons}}=0
\;\implies\;
\mathcal{L}_{\mathrm{dyn}}=0,
\qquad
\mathcal{L}_{\mathrm{dyn}}=0
\;\notimplies \;
\mathcal{L}_{\mathrm{lat\text{-}cons}}=0.
\]    
\end{proposition}

\begin{proof}
We prove each implication in turn.

\medskip\noindent\textbf{(i) \(\mathcal{L}_{\mathrm{lat\text{-}cons}}=0\implies\mathcal{L}_{\mathrm{dyn}}=0\).}
By definition of \(\mathcal{L}_{\mathrm{lat\text{-}dyn}}\) in \eqref{eq:loss-lat-dyn},  
\[
\mathcal{L}_{\mathrm{lat\text{-}cons}}=0
\quad\Longrightarrow\quad
\hat{\mathbf{z}}_{l|k}
=\mathcal{E}_\theta(\mathbf{x}_{k+l})
\quad\forall\,k,l.
\]
Applying the perfect‐reconstruction assumption of the AE we get,
\[
\mathcal{D}_\phi\bigl(\hat{\mathbf{z}}_{l|k}\bigr)
=\mathcal{D}_\phi\bigl(\mathcal{E}_\theta(\mathbf{x}_{k+l})\bigr)
=\mathbf{x}_{k+l}.
\]
Hence each term in \eqref{eq:loss-dyn} vanishes, and so \(\mathcal{L}_{\mathrm{dyn}}=0\).

\medskip\noindent\textbf{(ii) \(\mathcal{L}_{\mathrm{dyn}}=0\notimplies\mathcal{L}_{\mathrm{lat\text{-}cons}}=0\).} 
Suppose \(\mathcal{L}_{\mathrm{dyn}}=0\).  Then by \eqref{eq:loss-dyn},
\[
\mathcal{D}_\phi\bigl(\hat{\mathbf{z}}_{l|k}\bigr)
=\mathbf{x}_{k+l}
\quad\forall\,k,l.
\]
However, since \(r\ll n\) the decoder \(\mathcal{D}_\phi\) is generally many‑to‑one and need not be injective off the encoder’s image.  Thus there may exist
\(\hat{\mathbf{z}}_{l|k}\neq\mathcal{E}_\theta(\mathbf{x}_{k+l})\)
such that
\(\mathcal{D}_\phi(\hat{\mathbf{z}}_{l|k})=\mathbf{x}_{k+l}\).
In such a case \(\mathcal{L}_{\mathrm{dyn}}=0\), yet
\(\|\mathcal{E}_\theta(\mathbf{x}_{k+l})-\hat{\mathbf{z}}_{l|k}\|_2^2>0\)
for some \(k,l\), so \(\mathcal{L}_{\mathrm{lat\text{-}dyn}}>0\).
\end{proof}

In many engineering applications, the input vector $\mathbf{u}_k$ is high-dimensional, for example when actuation is distributed in space or represented by boundary fields.
Similar to the dimensionality reduction of the state vector $\mathbf{x}_k$, it is often advantageous to learn a compact latent representation for the inputs as well.
To this end, we introduce a latent input vector $\mathbf{u}^\prime \in \mathbb{R}^{m^\prime}$ with $m^\prime \ll m$, obtained using a dedicated input encoder.
This additional task could be easily integrated into our end-to-end framework to learn control-affine latent dynamics.
Specifically, we augment the multi-objective loss function given in \eqref{eq:mult-obj-loss} with a control autoencoder loss term $\mathcal{L}_{\mathrm{C-AE}}$ weighted by $\gamma_4$, which reads as
\begin{equation}
    \label{eq:control-AE}
    \mathcal{L}_{\mathrm{C-AE}} = \frac{1}{N} \sum_{k=0}^{N-1} \| \mathbf{u}_k - \mathcal{D}^\prime_{\phi^\prime}(\mathcal{E}^\prime_{\theta^\prime}(\mathbf{u}_k)) \|_2^2,
\end{equation}
where $\mathcal{E}^\prime$ and $\mathcal{D}^\prime$ represent the input encoder and decoder, parameterized with $\theta^\prime$ and $\phi^\prime$, respectively.
After introducing the latent input representation, the control-affine latent dynamics in \eqref{eq:latent-affine} and the rollout model in \eqref{eq:rollouts} are expressed in terms of $\mathbf{u}^\prime_k=\mathcal{E}^\prime_{\theta^\prime}(\mathbf{u}_k)$ instead of the high-dimensional input vector $\mathbf{u}_k$, i.e.,
\begin{equation}
\label{eq:rollouts-with-input-encoder}
\hat{\mathbf{z}}_{0|k} := \mathcal{E}_\theta(\mathbf{x}_k), \quad
\hat{\mathbf{z}}_{l+1|k} := \mathbf{a}_{\varphi_1}(\hat{\mathbf{z}}_{l|k}) + \mathbf{B}_{\varphi_2}(\hat{\mathbf{z}}_{l|k})\cdot \mathcal{E}^\prime_{\theta^\prime}(\mathbf{u}_{k+l}), \quad l = 0,\dots,M-1.
\end{equation}
All architectural dimensions for the drift and input networks are adjusted accordingly.

In this work, we propose training the ROM using a two-stage strategy.
First, the AE components are pretrained for few epochs, ensuring that the latent representation for the states (and input, if applicable) are well-initialized.
Subsequently, the full ROM is trained by minimizing \eqref{eq:mult-obj-loss} until convergence using the pretrained AE parameters as initializations and manually tuning the loss weights $\gamma_i$.
We note that the model's performance is sensitive to the choice of these weighting coefficients, which must be tuned properly.
Alternatively, more sophisticated multi-objective optimization algorithms could be used instead of manual tuning such as the self-adaptive method \citep{mcclenny2023self} or gradient normalization \citep{chen2018gradnorm}.

\section{Sequence-based control-affine latent dynamics}
\label{sec:seq-based-affine}
Predicting the future discrete (latent) state vector based solely on the current one is very challenging. 
Accordingly, autoregressive models typically address this by processing a sequence of past states to predict future values, which are maintained using a memory mechanism for subsequent evaluations.
We therefore adapt our dynamical model to operate on sequences of latent vectors and control inputs, while preserving the control-affine structure.
To this end, we introduce the extended state vector $\boldsymbol{\xi}_k$ defined as:
\begin{equation}
    \label{eq:extended-state}
    \boldsymbol{\xi}_k = [\mathbf{z}^T_{k-H}, \cdots, \mathbf{z}^T_{k-1}, \mathbf{u}^T_{k-H}, \cdots, \mathbf{u}^T_{k-1},\mathbf{z}_k^T]^T,
\end{equation}
where $H$ is a hyperparameter denoting the sequence length of previous latent vectors and inputs. 
By incorporating past inputs alongside past states, the extended state vector provides a complete information set for both delay-free and time-delay systems, making the approach applicable to a strictly broader class of dynamical systems.
We highlight that the current input $\mathbf{u}_k$ is not defined in $\boldsymbol{\xi}_k$ because we seek a control-affine model in the form
\begin{equation}
    \label{eq:affine-extended}
    \boldsymbol{\xi}_{k+1} = \boldsymbol{\alpha}(\boldsymbol{\xi}_k) + \boldsymbol{\beta}(\boldsymbol{\xi}_k)\cdot\mathbf{u}_k,
\end{equation}
where $\boldsymbol{\alpha}:\mathbb{R}^{d} \rightarrow \mathbb{R}^{d}$ and $\boldsymbol{\beta}:\mathbb{R}^{d} \rightarrow \mathbb{R}^{d \times m}$, with $d=(H+1)r + H m$.
According to \eqref{eq:extended-state}, the function $\boldsymbol{\alpha}$ acts primarily as a linear shifting operator that encodes the memory update, while also incorporating a nonlinear drift term for updating $\mathbf{z}_k$, similar to \eqref{eq:latent-affine}. Thus, we structure $\boldsymbol{\alpha}(\boldsymbol{\xi}_k)$ as:
\begin{equation}
\label{eq:alpha}
\boldsymbol{\alpha}(\boldsymbol{\xi}_k) = \mathbf{S} \boldsymbol{\xi}_k + \begin{bmatrix} \mathbf{0}_{(H r + H m)} \\ \widetilde{\mathbf{a}}_{\varphi_1}(\boldsymbol{\xi}_k) \end{bmatrix},
\end{equation}
where $\mathbf{S} \in \mathbb{R}^{d \times d}$ is a block-shift matrix that advances the memory sequences, and $\widetilde{\mathbf{a}}_{\varphi_1}:\mathbb{R}^d \rightarrow\mathbb{R}^r$ is a neural network parameterized by $\varphi_1$.
On the other hand, $\boldsymbol{\beta}$ models the influence of the input on the latent dynamics through the extended vector $\boldsymbol{\xi}_k$. It populates the memory buffer with the current input $\mathbf{u}_k$ and applies the control matrix to the latent state update:
\begin{equation}
\label{eq:beta}
\boldsymbol{\beta}(\boldsymbol{\xi}_k) =
\begin{bmatrix}
\mathbf{0}_{Hr \times m} \\
\mathbf{0}_{(H-1)m \times m} \\
\mathbf{I}_{m \times m} \\
\widetilde{\mathbf{B}}_{\varphi_2}(\boldsymbol{\xi}_k)
\end{bmatrix} \in \mathbb{R}^{d \times m},
\end{equation}
where $\widetilde{\mathbf{B}}_{\varphi_2}:\mathbb{R}^d \rightarrow \mathbb{R}^{r \times m}$ is a neural network parameterized by $\varphi_2$.

\begin{proposition}
The extended state transition \eqref{eq:affine-extended} governed by $\boldsymbol{\alpha}$ and $\boldsymbol{\beta}$ defined in \eqref{eq:alpha} and \eqref{eq:beta} constitutes a control-affine representation of the sequence-based dynamics.
\end{proposition}

\begin{proof}
We verify by direct substitution that $\boldsymbol{\alpha}(\boldsymbol{\xi}_k)+\boldsymbol{\beta}(\boldsymbol{\xi}_k)\mathbf{u}_k$ recovers $\boldsymbol{\xi}_{k+1}$ entry-by-entry.

The linear shift matrix $\mathbf{S} \in \mathbb{R}^{d \times d}$ in \eqref{eq:alpha} is constructed to shift the state and input sequences forward by one time step.
We partition $\boldsymbol{\xi}_k$ into the state history $\mathbf{Z}_{k} = [\mathbf{z}^T_{k-H}, \dots, \mathbf{z}^T_{k-1}]^T \in \mathbb{R}^{Hr}$, the input history $\mathbf{U}_{k} = [\mathbf{u}^T_{k-H}, \dots, \mathbf{u}^T_{k-1}]^T \in \mathbb{R}^{Hm}$, and the current state $\mathbf{z}_k \in \mathbb{R}^r$. The matrix $\mathbf{S}$ takes the block-diagonal form:
\begin{equation*}
\mathbf{S} = \begin{bmatrix}
\mathbf{S}_z & \mathbf{0}_{Hr \times Hm} & \mathbf{E}_z \\
\mathbf{0}_{Hm \times Hr} & \mathbf{S}_u & \mathbf{0}_{Hm \times r} \\
\mathbf{0}_{r \times Hr} & \mathbf{0}_{r \times Hm} & \mathbf{0}_{r \times r}
\end{bmatrix},
\end{equation*}
where $\mathbf{S}_z \in \mathbb{R}^{Hr \times Hr}$ and $\mathbf{S}_u \in \mathbb{R}^{Hm \times Hm}$ are block up-shift matrices (with identity matrices $\mathbf{I}_{r\times r}$ and $\mathbf{I}_{m \times m}$ on their upper sub-diagonals, respectively, and zeros elsewhere), and $\mathbf{E}_z \in \mathbb{R}^{Hr \times r}$ contains $\mathbf{I}_{r \times r}$ in its bottom block row to move $\mathbf{z}_k$ into the history buffer. Applying $\mathbf{S}$ yields:
\begin{equation}
\label{eq:shifted}
\mathbf{S} \boldsymbol{\xi}_k = \begin{bmatrix} \mathbf{z}_{k-H+1} \\ \vdots \\ \mathbf{z}_{k} \\ \mathbf{u}_{k-H+1} \\ \vdots \\ \mathbf{u}_{k-1} \\ \mathbf{0}_{m} \\ \mathbf{0}_{r} \end{bmatrix}.
\end{equation}
Substituting \eqref{eq:shifted} into \eqref{eq:alpha} and applying \eqref{eq:affine-extended}, we obtain the full update:
\begin{equation}
\boldsymbol{\xi}_{k+1} =
\underbrace{\begin{bmatrix} \mathbf{z}_{k-H+1} \\ \vdots \\ \mathbf{z}_{k} \\ \mathbf{u}_{k-H+1} \\ \vdots \\ \mathbf{u}_{k-1} \\ \mathbf{0}_m \\ \widetilde{\mathbf{a}}_{\varphi_1}(\boldsymbol{\xi}_k) \end{bmatrix}}_{\boldsymbol{\alpha}(\boldsymbol{\xi}_k)}
+
\underbrace{\begin{bmatrix} \mathbf{0}_{r \times m} \\ \vdots \\ \mathbf{0}_{r \times m} \\ \mathbf{0}_{m \times m} \\ \vdots \\ \mathbf{0}_{m \times m} \\ \mathbf{I}_{m \times m} \\ \widetilde{\mathbf{B}}_{\varphi_2}(\boldsymbol{\xi}_k) \end{bmatrix}}_{\boldsymbol{\beta}(\boldsymbol{\xi}_k)} \mathbf{u}_k
=
\begin{bmatrix} \mathbf{z}_{k-H+1} \\ \vdots \\ \mathbf{z}_{k} \\ \mathbf{u}_{k-H+1} \\ \vdots \\ \mathbf{u}_{k-1} \\ \mathbf{u}_k \\ \widetilde{\mathbf{a}}_{\varphi_1}(\boldsymbol{\xi}_k) + \widetilde{\mathbf{B}}_{\varphi_2}(\boldsymbol{\xi}_k)\mathbf{u}_k \end{bmatrix}.
\end{equation}
The first $H$ block rows correctly represent the shifted latent state history, ending with $\mathbf{z}_k$. The subsequent $H$ block rows represent the shifted input history, where the term $\mathbf{I}_{m \times m} \mathbf{u}_k$ from $\boldsymbol{\beta}$ correctly inserts the current input into the final position of the memory buffer. Finally, the last block row yields the latent state update $\mathbf{z}_{k+1} := \widetilde{\mathbf{a}}_{\varphi_1}(\boldsymbol{\xi}_k) + \widetilde{\mathbf{B}}_{\varphi_2}(\boldsymbol{\xi}_k)\mathbf{u}_k$. Because $\boldsymbol{\alpha}$ and $\boldsymbol{\beta}$ depend solely on $\boldsymbol{\xi}_k$ and act linearly with respect to $\mathbf{u}_k$, the transition is control-affine in $(\boldsymbol{\xi}_k, \mathbf{u}_k)$.
\end{proof}

The sequence-based control affine ROM could be trained using the same loss function given in \eqref{eq:mult-obj-loss}, with adapting the neural networks of the control-affine model to operate on $\boldsymbol{\xi}_k$ instead of $\mathbf{z}_k$. 
We note that the encoder network processes every high-dimensional state vector independently (see Algorithm~\ref{alg:rom-training}).

When an input AE is employed, the sequence-based formulation is modified in a straightforward manner, where the input vector $\mathbf{u}_k \in \mathbb{R}^{m}$ is replaced by $\mathbf{u}^\prime_k=\mathcal{E}^\prime_{\theta^\prime}(\mathbf{u}_k) \in \mathbb{R}^{m^\prime}$ and the dimensions are adjusted accordingly.
The full training procedure is summarized in Algorithm~\ref{alg:rom-training}.
Once the ROM is trained, it can be deployed to simulate system dynamics for a new set of initial conditions and input signals.
The full inference procedure is illustrated in Fig.~\ref{fig:framework} for the case with high-dimensional inputs and $\mathcal{X}=\mathbb{R}^n$. 
Note that the input decoder does not appear in Fig.~\ref{fig:framework} as the dynamics is predicted from the latent input sequence.
However, it is important to have a trained input decoder to have a seamless transformation between the physical and latent input space, which is relevant for control applications where inputs must be applied in the physical space.
For the case with low-dimensional input vector, the input autoencoder is omitted, and the latent dynamics is predicted directly from the physical input sequence.

\begin{algorithm}[!ht]
\caption{Training of Sequence-Based Control-Affine ROM (with optional Input Autoencoder)}
\label{alg:rom-training}
\begin{algorithmic}[1]
\REQUIRE Dataset $\{\mathbf{x}_k,\mathbf{u}_k\}_{k=0}^{N-1}$, rollout steps $M$, history length $H$, 
number of pretraining epochs $E_1$, number of ROM training epochs $E_2$, 
weighting coefficients $\gamma_1,\gamma_2,\gamma_3,(\gamma_4)$
\ENSURE Trained parameters $\theta,\phi,\varphi_1,\varphi_2,( \theta',\phi')$
\vspace{0.5em}

\STATE \textbf{Stage 1: Autoencoder Pretraining}
\STATE Randomly initialize $\theta,\phi$ and, if used, $\theta',\phi'$
\FOR{epoch $=1,\cdots,E_1$}
  \FOR{$k = 0,\cdots,N-1$}
    \STATE $\mathbf{z}_k \gets \mathcal{E}_\theta(\mathbf{x}_k)$
    \STATE $\hat{\mathbf{x}}_k \gets \mathcal{D}_\phi(\mathbf{z}_k)$
    \STATE Compute state AE loss $\mathcal{L}_{\text{AE}} = \|\hat{\mathbf{x}}_k - \mathbf{x}_k\|_{\mathcal{X}}^2$
    \IF{input AE is used}
      \STATE $\mathbf{u}'_k \gets \mathcal{E}'_{\theta'}(\mathbf{u}_k)$
      \STATE $\hat{\mathbf{u}}_k \gets \mathcal{D}'_{\phi'}(\mathbf{u}'_k)$
      \STATE Compute control AE loss $\mathcal{L}_{\text{C-AE}} = \|\hat{\mathbf{u}}_k - \mathbf{u}_k\|_2^2$
    \ENDIF
    \STATE Update AE parameters using gradient descent
  \ENDFOR
\ENDFOR
\vspace{0.5em}

\STATE \textbf{Stage 2: ROM Training}
\STATE Initialize $(\theta,\phi,(\theta',\phi'))$ from Stage 1; randomly initialize $\varphi_1,\varphi_2$
\STATE Prepare sequential training samples 
$\{ \mathbf{x}_{k-H{:}k},\;\mathbf{u}_{k-H{:}k} \} 
\rightarrow 
\{ \mathbf{x}_{k+1{:}k+M} \}$.
\FOR{epoch $=1,\cdots,E_2$}
  \FOR{$k = H,\cdots,N-M$}
    \FOR{$i = 0,\cdots,H$}
      \STATE $\mathbf{z}_{k-i} \gets \mathcal{E}_\theta(\mathbf{x}_{k-i})$
    \ENDFOR
    \IF{input AE is used}
      \FOR{$i = 0,\cdots,H$}
        \STATE $\mathbf{u}'_{k-i} \gets \mathcal{E}'_{\theta'}(\mathbf{u}_{k-i})$
      \ENDFOR
      \STATE Construct $\boldsymbol{\xi}_{0|k}$ using $\{\mathbf{z}_{k-i},\mathbf{u}'_{k-i}\}$
    \ELSE
      \STATE Construct $\boldsymbol{\xi}_{0|k}$ using $\{\mathbf{z}_{k-i},\mathbf{u}_{k-i}\}$
    \ENDIF
    \STATE $\hat{\mathbf{z}}_{0|k} \gets \mathcal{E}_\theta(\mathbf{x}_k)$
    \FOR{$l = 0,\cdots,M-1$}
      \IF{input AE is used}
        \STATE $\mathbf{u}'_{k+l} \gets \mathcal{E}'_{\theta'}(\mathbf{u}_{k+l})$
        \STATE $\hat{\mathbf{z}}_{l+1|k} \gets 
        \widetilde{\mathbf{a}}_{\varphi_1}(\boldsymbol{\xi}_{l|k}) 
        + 
        \widetilde{\mathbf{B}}_{\varphi_2}(\boldsymbol{\xi}_{l|k}) \cdot \mathbf{u}'_{k+l}$
      \ELSE
        \STATE $\hat{\mathbf{z}}_{l+1|k} \gets 
        \widetilde{\mathbf{a}}_{\varphi_1}(\boldsymbol{\xi}_{l|k}) 
        + 
        \widetilde{\mathbf{B}}_{\varphi_2}(\boldsymbol{\xi}_{l|k}) \cdot \mathbf{u}_{k+l}$
      \ENDIF
      \STATE Update $\boldsymbol{\xi}_{l+1|k}$ with new latent state $\hat{\mathbf{z}}_{l+1|k}$
    \ENDFOR
    \STATE Compute total loss $\mathcal{L}$ as in \eqref{eq:mult-obj-loss}
    \STATE Update parameters $\theta,\phi,\varphi_1,\varphi_2,( \theta',\phi')$ using gradient descent
  \ENDFOR
\ENDFOR
\end{algorithmic}
\end{algorithm}

\begin{figure}[t]
    \centering
    \includegraphics[]{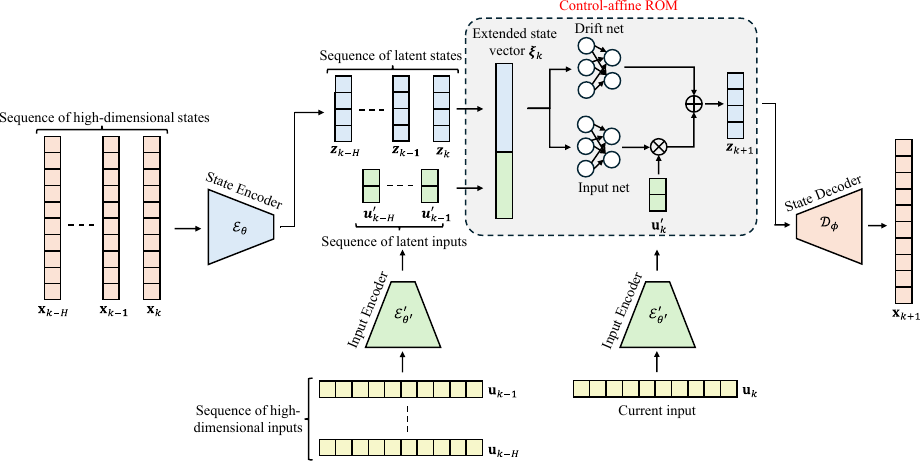}
    \caption{Single-step inference of the control-affine ROM. The encoders compress the high-dimensional state and input sequences into low-dimensional latent ones, which are combined to form the extended state vector $\boldsymbol{\xi}_k$ as in \eqref{eq:extended-state}. The future latent dynamics is predicted using a control-affine structure in terms of the current latent input $\mathbf{u}^\prime_k$. The state decoder maps the predicted latent vector to the corresponding high-dimensional representation.}
    \label{fig:framework}
\end{figure}

\section{Feedback linearization of Control-affine systems}
\label{sec:FL}
One of the main advantages of constraining the dynamics to a control-affine structure is the ability to apply feedback linearization for system control.
The main idea of feedback linearization is to algebraically transform the nonlinear dynamics into an equivalently linear representation \citep{isidori1995nonlinear}.
This allows the use of well-established linear methods to perform stabilization or tracking.

Consider the control-affine discrete system given in \eqref{eq:latent-affine}, the goal of feedback linearization is to obtain a model in the form 
\begin{equation}
    \label{eq:feedback-linear-goal}
    \mathbf{z}_{k+1} = \bar{\mathbf{A}}\mathbf{z}_{k} + \bar{\mathbf{B}}\mathbf{v}_k,
\end{equation}
where $\bar{\mathbf{A}} \in \mathbb{R}^{r \times r}$ and $\bar{\mathbf{B}} \in \mathbb{R}^{r \times m}$ denote the desired system and input matrices defining the behavior of the linearized system, and $\mathbf{v}_k\in \mathbb{R}^m$ is the new virtual input vector.
Within the ROM framework presented in this paper, such linearization is attainable under the condition stated in Proposition~\ref{prop:cond-FL}.


\begin{proposition}
\label{prop:cond-FL}
A full-state feedback linearization of the control-affine latent dynamics given in \eqref{eq:latent-affine} exists if $\mathbf{B}_{\varphi_2}(\mathbf{z}_k)$ has full row rank, i.e., $\text{rank}(\mathbf{B}_{\varphi_2}(\mathbf{z}_k))=r$ for all $\mathbf{z}_k \in \mathcal{D} \subset \mathbb{R}^r$, where $\mathcal{D}$ denotes the domain of interest.
Under this condition, the linearizing feedback law is given by
\begin{equation}
    \label{eq:linearizing-law}
    \mathbf{u}_k = \mathbf{B}^{R}_{\varphi_2}(\mathbf{z}_k)(\bar{\mathbf{A}}\mathbf{z}_{k} + \bar{\mathbf{B}}\mathbf{v}_k - \mathbf{a}_{\varphi_1}(\mathbf{z}_k)),
\end{equation}
where $\mathbf{B}^{R}_{\varphi_2}$ denote the right pseudoinverse of $\mathbf{B}_{\varphi_2}$.
\end{proposition}

\begin{proof}
Substituting \eqref{eq:linearizing-law} into the original system \eqref{eq:latent-affine} yields:
\begin{equation*}
\label{eq:proof-FL}
    \mathbf{z}_{k+1}=\mathbf{a}_{\varphi_1}(\mathbf{z}_k)+\mathbf{B}_{\varphi_2}(\mathbf{z}_k)\mathbf{B}^{R}_{\varphi_2}(\mathbf{z}_k)(\bar{\mathbf{A}}\mathbf{z}_{k} + \bar{\mathbf{B}}\mathbf{v}_k - \mathbf{a}_{\varphi_1}(\mathbf{z}_k)) = \bar{\mathbf{A}}\mathbf{z}_{k} + \bar{\mathbf{B}}\mathbf{v}_k,
\end{equation*}
by construction of $\mathbf{u}_k$.
For a matrix $\mathbf{B}^{R}_{\varphi_2}(\mathbf{z}_k)$ to exist such that $\mathbf{B}_{\varphi_2}(\mathbf{z}_k)\mathbf{B}^{R}_{\varphi_2}(\mathbf{z}_k)=\mathbf{I}_{r \times r}$, the column space of $\mathbf{B}_{\varphi_2}(\mathbf{z}_k)$ must span $\mathbb{R}^r$, or equivalently $\text{rank}(\mathbf{B}_{\varphi_2}(\mathbf{z}_k))=r$ for all $\mathbf{z}_k \in \mathcal{D}$.
The condition $\text{rank}(\mathbf{B}_{\varphi_2}(\mathbf{z}_k))=r$ implies $r \le m$ for $\mathbf{B}_{\varphi_2}(\mathbf{z}_k) \in \mathbb{R}^{r \times m}$.
\end{proof}
For the special case $r=m$, $\mathbf{B}^{R}_{\varphi_2}(\mathbf{z}_k)$ reduces to the standard matrix inverse $\mathbf{B}^{-1}_{\varphi_2}(\mathbf{z}_k)$.
When $r<m$, the right pseudoinverse takes the form
\begin{equation}
    \label{eq:r<m}
    \mathbf{B}^{R}_{\varphi_2}(\mathbf{z}_k)=\mathbf{B}^{T}_{\varphi_2}(\mathbf{z}_k)(\mathbf{B}_{\varphi_2}(\mathbf{z}_k)\mathbf{B}^{T}_{\varphi_2}(\mathbf{z}_k))^{-1},
\end{equation}
which corresponds to the Moore-Penrose pseudoinverse.
In this scenario, the feedback linearization law is not unique and \eqref{eq:r<m} yields the minimal norm solution.

Since $r$ is a tunable hyperparameter of our ROM, setting $r\le m$ enables the possibility of full-state feedback linearization.
However, one has to ensure/verify that $\mathbf{B}_{\varphi_2}(\mathbf{z}_k)$ has full row rank within a region of interest.
We note that it is possible to encourage the invertibility of $\mathbf{B}_{\varphi_2}$ during training of the ROM by proper design of the input net, e.g., output a lower-triangular matrix with strictly positive diagonal entries, or adding an additional term to the loss function that penalizes small singular values.
In practice, setting $r \le m$ is often not viable, as it might reduce the ROM accuracy, especially when $m$ is small.
In such cases, full-state feedback linearization with \eqref{eq:linearizing-law} is not possible, as $\text{rank}(\mathbf{B}_{\varphi_2}(\mathbf{z}_k)) \le m <r$, and only partial-feedback linearization is possible, limiting the control to an $m-$dimensional subspace.

The ROM presented in this work employs a sequence-based control-affine latent dynamics as presented in Sec.~\ref{sec:seq-based-affine}.
Therefore, the feedback linearization should be applied to the control-affine latent model in \eqref{eq:affine-extended} defined with the extended state vector $\boldsymbol{\xi}_k \in \mathbb{R}^d$.
Unlike the standard model in \eqref{eq:latent-affine}, full-state feedback linearization is structurally impossible for the extended model.
This limitation arises from the definition of the input matrix $\boldsymbol{\beta}(\boldsymbol{\xi}_k)$ in \eqref{eq:beta}, which contains an identity block $\mathbf{I}_{m \times m}$ that shifts the current input into the history buffer.
Any nonlinear feedback law applied to $\mathbf{u}_k$ will unavoidably inject nonlinear combinations of $\boldsymbol{\xi}_k$ into the input history sequence of the extended state, preventing the entire vector $\boldsymbol{\xi}_{k+1}$ from evolving linearly.
However, we can achieve exact input/output feedback linearization by treating the current latent state as system output, i.e., $\mathbf{y}_k :=\mathbf{z}_k$.
We formalize this in the following proposition.
\begin{proposition}
    \label{prop:cond-FL-extended}
    Consider the sequence-based control-affine model \eqref{eq:affine-extended} with $\boldsymbol{\alpha}$ and $\boldsymbol{\beta}$ defined in \eqref{eq:alpha} and \eqref{eq:beta}, respectively, and let $\mathbf{y}_k=\mathbf{C}\boldsymbol{\xi}_k=\mathbf{z}_k$ define the output equation, where $\mathbf{C}\in \mathbb{R}^{r \times d}$ extracts the last $r$ components of $\boldsymbol{\xi}_k$.
    Exact input-output feedback linearization from a virtual input $\mathbf{v}_k \in \mathbb{R}^m$ to the subsequent output $\mathbf{y}_{k+1} \in \mathbb{R}^r$ is achievable if $\text{rank}(\widetilde{\mathbf{B}}_{\varphi_2}(\boldsymbol{\xi}_k))=r$ for all $\boldsymbol{\xi}_k \in \mathcal{D} \subset \mathbb{R}^d$, where $\mathcal{D}$ denotes the domain of interest.
    Under this condition, the linearizing feedback law is given by
    \begin{equation}
        \label{eq:linearizing-law-extended}
        \mathbf{u}_k = \widetilde{\mathbf{B}}^{R}_{\varphi_2}(\boldsymbol{\xi}_k) \left( \bar{\mathbf{A}}\mathbf{y}_{k} + \bar{\mathbf{B}}\mathbf{v}_k - \tilde{\mathbf{a}}_{\varphi_1}(\boldsymbol{\xi}_k) \right),
    \end{equation}
    where $\widetilde{\mathbf{B}}^{R}_{\varphi_2}$ denotes the right pseudoinverse of $\widetilde{\mathbf{B}}_{\varphi_2}$, and $\bar{\mathbf{A}}\in \mathbb{R}^{r \times r}$, $\bar{\mathbf{B}} \in \mathbb{R}^{r \times m}$ define the desired latent output dynamics:
    \[
    \mathbf{y}_{k+1} = \bar{\mathbf{A}}\mathbf{y}_k + \bar{\mathbf{B}}\mathbf{v}_k.
    \]
\end{proposition}

\begin{proof}
From the definition of the extended dynamics in \eqref{eq:affine-extended}--\eqref{eq:beta}, and the definition of the output at the next time step, i.e., $\mathbf{y}_{k+1}=\mathbf{C}\boldsymbol{\xi}_{k+1}$, the last \(r\) components of \(\boldsymbol{\xi}_{k+1}\) satisfy
\begin{equation}
    \label{eq:proof-in/out linearization}
    \mathbf{y}_{k+1} = \mathbf{z}_{k+1} = \widetilde{\mathbf{a}}_{\varphi_1}(\boldsymbol{\xi}_k) + \widetilde{\mathbf{B}}_{\varphi_2}(\boldsymbol{\xi}_k)\mathbf{u}_k.
\end{equation}
Substituting \eqref{eq:linearizing-law-extended} in \eqref{eq:proof-in/out linearization} yields:
\[ 
\mathbf{y}_{k+1} = \widetilde{\mathbf{a}}_{\varphi_1}(\boldsymbol{\xi}_k) + \widetilde{\mathbf{B}}_{\varphi_2}(\boldsymbol{\xi}_k)\widetilde{\mathbf{B}}^{R}_{\varphi_2}(\boldsymbol{\xi}_k) \left( \bar{\mathbf{A}}\mathbf{y}_{k} + \bar{\mathbf{B}}\mathbf{v}_k - \tilde{\mathbf{a}}_{\varphi_1}(\boldsymbol{\xi}_k) \right) = \bar{\mathbf{A}}\mathbf{y}_{k} + \bar{\mathbf{B}}\mathbf{v}_k, 
\]
by construction of $\mathbf{u}_k$. 
For a matrix $\widetilde{\mathbf{B}}^{R}_{\varphi_2}(\boldsymbol{\xi}_k)$ to exist such that $\widetilde{\mathbf{B}}_{\varphi_2}(\boldsymbol{\xi}_k)\widetilde{\mathbf{B}}^{R}_{\varphi_2}(\boldsymbol{\xi}_k)=\mathbf{I}_{r \times r}$, the column space of $\widetilde{\mathbf{B}}_{\varphi_2}(\boldsymbol{\xi}_k)$ must span $\mathbb{R}^r$, or equivalently $\text{rank}(\widetilde{\mathbf{B}}_{\varphi_2}(\boldsymbol{\xi}_k))=r$ for all $\boldsymbol{\xi}_k \in \mathcal{D}$.
The condition $\text{rank}(\widetilde{\mathbf{B}}_{\varphi_2}(\boldsymbol{\xi}_k))=r$ implies $r \le m$ for $\widetilde{\mathbf{B}}_{\varphi_2}(\boldsymbol{\xi}_k) \in \mathbb{R}^{r \times m}$.

\end{proof}

When an input AE is employed, propositions~\ref{prop:cond-FL} and ~\ref{prop:cond-FL-extended} remain valid, with the input dimension $m$ replaced by the latent input dimension $m^\prime$.
Once linearized, the model can be controlled using standard linear control methods, e.g., pole placement, PID, or LQR. 
Figure \ref{fig:Blockdiagram} shows a block diagram of this process, illustrating the input-output linearization of the control-affine ROM coupled with a PID controller for the linearized system.
\begin{figure}[t]
    \centering
    \includegraphics{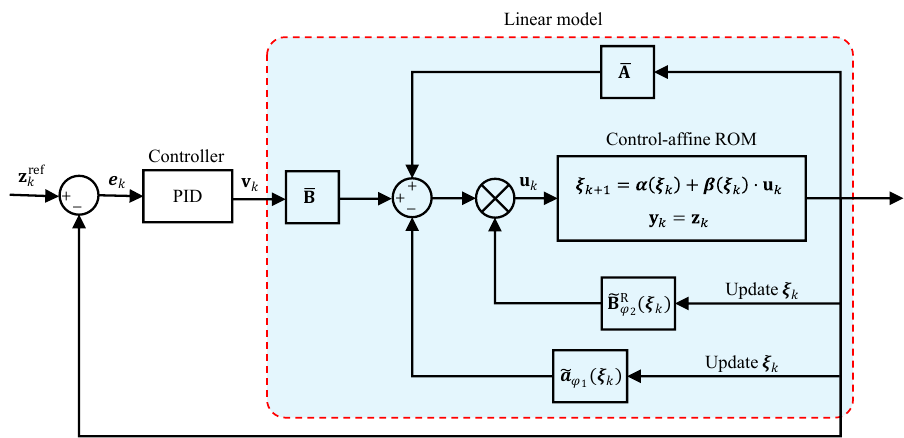}
    \caption{Input/Output feedback linearization block diagram. The linearized model is controlled using a PID controller to track the reference $\mathbf{z}^{\mathrm{ref}}_k$.}
    \label{fig:Blockdiagram}
\end{figure}

\section{Numerical examples}
\label{sec:examples}
\subsection{One-dimensional heat equation}
\label{sec:heat}
We consider first the one-dimensional heat equation as a demonstrative example for the proposed method
\begin{equation}
    \label{eq:1d-heat}
    \frac{\partial T(x,t)}{\partial t} = \alpha \frac{\partial^2 T(x,t)}{\partial x^2}+u(x,t), \quad \forall t \in [t_0,t_f], x\in \Omega,
\end{equation}
which describes the temperature evolution of the temperature field $T(x,t)$ over a one-dimensional domain $\Omega$.
Physically, this setup can be interpreted as a thin steel beam heated by a distributed heat source along its length (e.g., an electric resistance wire that runs through the beam), where $\alpha$ denotes the thermal diffusivity and $u(x,t)$ represents the heat source.
The equation is solved numerically on a uniform spatial grid $\Omega=[0,1]$ with $t_0=0$ and $t_f=0.5$ using an explicit finite-difference forward-time central-space scheme, where the time step $\Delta t$ is chosen according to the classical stability condition $\gamma=\alpha \Delta t/\Delta x^2 \le 0.5$ \citep{Press2007FTCS}.
Dirichlet boundary conditions are imposed at both ends of the domain, i.e., $T(0,t)=T(1,t)=0$.
The initial temperature distribution is defined as $T_0(x) = \sin{\pi x}$.

\subsubsection{Data-driven control-affine ROM}
A large number of simulations is generated to construct a training database for the control-affine model.
Every simulation is parameterized by a unique spatio-temporal input source term, which is modeled as a random superposition of $I$ temporally activated Gaussian distributions, i.e.,
\begin{equation}
\label{eq:input-heat}
\begin{aligned}
    &u(x,t) = \sum_{i=1}^I u_i(x,t), \\
    &u_i(x,t)=\begin{cases}
         A_i \exp\!\left(-\frac{(x-c_i)^2}{\sigma_i^2}\right), &\text{if } t\in [t_i^s, t_i^e],\\ 
         0, & \mathrm{otherwise}.
    \end{cases}
\end{aligned}
\end{equation}
The number of sources $I$ is randomly chosen between $1$ and $3$ per simulation.
The parameters of each Gaussian component $i$ are drown independently from uniform distributions: amplitude $A_i \sim \mathcal{U}(5.0, 20.0)$, spatial center $c_i \sim \mathcal{U}(0.2, 0.8)$, standard deviation $\sigma_i \sim \mathcal{U}(0.01, 0.1)$, activation start $t_i^s \sim \mathcal{U}(0.0, 0.3)$, and activation end $t_i^e \sim \mathcal{U}(0.3, t_f)$.
The resulting source term $u(x,t)$ is then normalized to $[0,1]$ using Min-Max scaling.

The temperature and input profiles over the discretized domain $\Omega$ define the high-dimensional state-input trajectories that are used to build the control-affine ROM.
In total, 1000 simulations are used for training, 400 for validation and hyperparameter tuning, and 600 for testing.
The control-affine ROM framework visualized in Fig.~\ref{fig:framework} is formulated such that $\mathbf{x}_k \in \mathbb{R}^n$ denotes the high-dimensional temperature profile over $\Omega$ for time step $t_k$, and $\mathbf{u}_k \in \mathbb{R}^m$ is the corresponding high-dimensional input vector obtained by discretizing \eqref{eq:input-heat} on the same spatial grid.
Here, $n=m$ denotes the number of grid points resulting from the uniform discretization of $\Omega$.
Since both the state and the input are high-dimensional fields, an autoencoder is employed for each of them.
Accordingly, both state and input encoders are designed to map the respective high-dimensional vectors into latent representation.
In order to enable feedback linearization (see Sec.~\ref{sec:FL}), the control-affine ROM should be designed such that the latent dimensions satisfies $r\le m^\prime$, where $r$ denotes the state latent dimension and $m^\prime$ is the input latent dimension, which is a necessary condition for the input matrix to attain full row rank (cf. Proposition~\ref{prop:cond-FL-extended}).
In this example, we set $r=m^\prime=6$, yielding a square input matrix.
This allows the full-rank condition to be interpreted as invertibility.
The architectures of the state and input autoencoders are shown in Tab.~\ref{tab:encoder-heat}.
Both share the same structure except for the output layers, where the input encoder uses a sigmoid activation.
In both cases, the high-dimensional state and input fields are mapped to their latent representations through a sequence of fully connected layers with ReLU activations, in which the number of neurons is progressively reduced down to $r$ units for the state encoder and $m^\prime$ units for the input encoder at the output.
The decoders adopt symmetric architectures to reconstruct the high-dimensional vectors, with a sigmoid activation in the output layer of the input decoder.
This design choice is consistent with the scaling of the input $u(x,t)$ to $[0,1]$
and facilitates stable control by ensuring that the reconstructed control signals remain within the range observed during training.
In addition, the drift and input networks each consist of 2 fully connected layers with 128 neurons and ReLU activation, followed by a linear output layer.
Specifically, the output layer of the drift network $\widetilde{\mathbf{a}}_{\varphi_1}$ has dimension $r$, while the input network $\widetilde{\mathbf{B}}_{\varphi_2}$ outputs a vector of dimension $r\cdot m^\prime$, which is reshaped into a matrix of size $r \times m^\prime$.
The reported hyperparameters were tuned manually to obtain the best predictive performance on the validation data.
\begin{table}[t]
\centering
\caption{Architecture of the state and input encoders for the 1D heat equation example.}
\label{tab:encoder-heat}
\begin{tabular}{l|l}
\hline
State encoder & Input encoder \\
\hline
input: $101$ & input: $101$ \\
Dense(64), ReLU & Dense(64), ReLU \\
Dense(32), ReLU & Dense(32), ReLU \\
Dense($r$), linear & Dense($m^\prime$), Sigmoid \\
\hline
\end{tabular}
\end{table}

The control-affine ROM is trained following Algorithm~\ref{alg:rom-training} using the ADAM optimizer \citep{Kingma2014} with initial learning rate $10^{-3}$, which is reduced by a factor of two if the validation loss does not improve for 25 consecutive epochs.
First, the autoencoders are pre-trained for 10 epochs to obtain stable latent representations.
Subsequently, all model components are trained jointly for 500 epochs using sequences of length $H=9$ and a prediction horizon of length $M=5$.
The weighting coefficients of the multi-objective loss function in \eqref{eq:mult-obj-loss} are set to $\gamma_1=\gamma_2=\gamma_4=1$ and $\gamma_3=0.3$.
\begin{figure}[t]
    \centering
    \includegraphics[width=\textwidth]{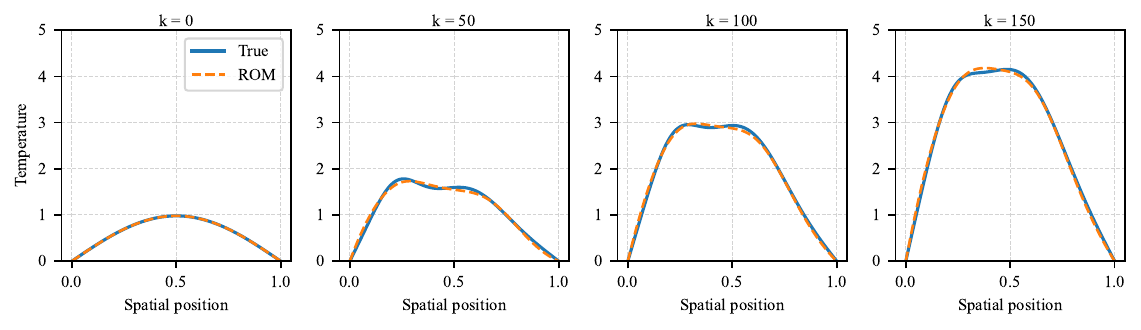}
    \caption{Comparison of ground-truth temperature profiles and ROM-predicted outputs at four representative time steps for a test simulation of the heat equation.}
    \label{fig:heat_recons}
\end{figure}
The model is evaluated recursively using a randomly chosen input trajectory obtained from the test dataset.
The predicted temperature profiles at four representative time steps are shown in Fig.~\ref{fig:heat_recons}.
The control-affine ROM exhibits excellent agreement with the ground-truth solution across all time steps.
Quantitatively, this is reflected by a root-mean-square error (RMSE) of $2.67 \times 10^{-3}$, computed over all spatial points and time steps in the simulation horizon.
This low error indicates that the model accurately captures both the spatial structure and temporal evolution of the temperature field.
\begin{figure}[t]
    \centering
    \includegraphics[width=\textwidth]{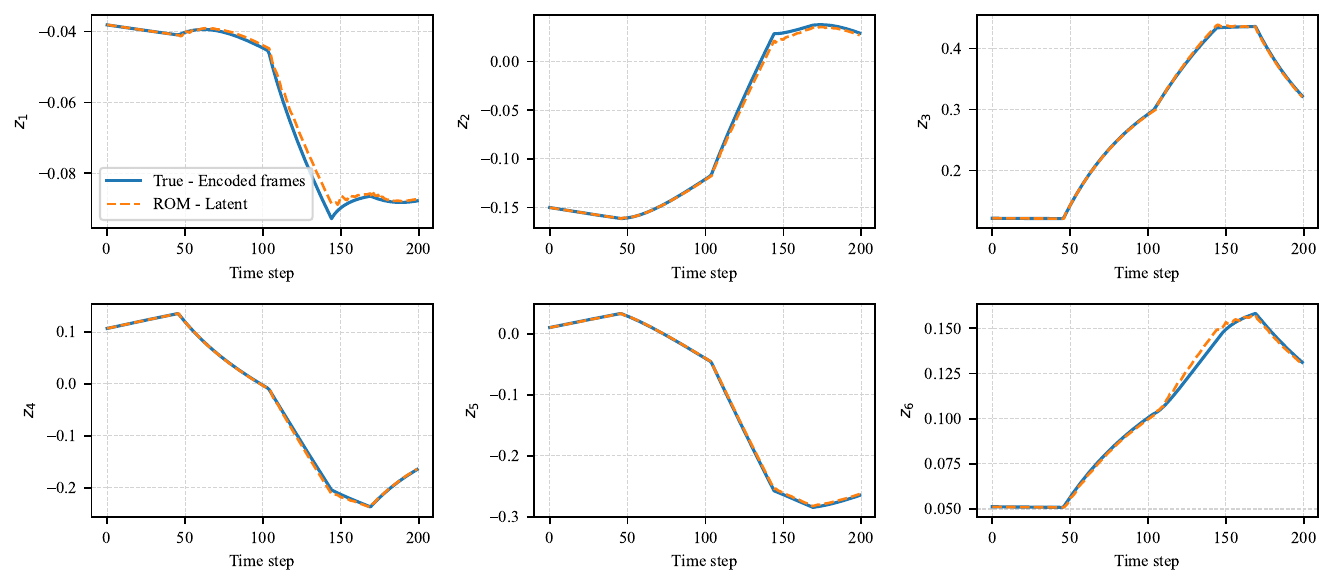}
    \caption{Comparison of the latent dynamics predicted by the control-affine model and the corresponding ground-truth latent trajectories obtained via state encoding for a test simulation of the heat equation. Each subplot shows one component $\mathrm{z}_i$ of the latent state vector $\mathbf{z} \in \mathbb{R}^r$, with reduced dimension $r=m^\prime=6$.}
    \label{fig:eval_latent_heat}
\end{figure}

The latent-space dynamics are evaluated in Fig.~\ref{fig:eval_latent_heat}, where the predictions are obtained using the control-affine model in~\eqref{eq:affine-extended}.
Specifically, initial sequences of high-dimensional temperature fields and input signals of length $H+1=10$ are randomly selected from the test dataset.
These sequences are encoded using the state and input AEs, yielding the latent representations $\mathbf{z}_k$ and $\mathbf{u}_k^\prime$, respectively.
Based on these encoded sequences, the initial extended state vector is constructed according to~\eqref{eq:extended-state} as 
\[
    \boldsymbol{\xi}_{10}=[\mathbf{z}_0^T,\ldots,\mathbf{z}_9^T,\mathbf{u}_{0}^{\prime^T},\ldots,\mathbf{u}_9^{\prime^T},\mathbf{z}_{10}^T]^T.
\]
This vector serves as the initial condition for the latent dynamical system and is provided as input to the drift $\widetilde{\mathbf{a}}_{\varphi_1}(\boldsymbol{\xi}_k)$ and input $\widetilde{\mathbf{B}}_{\varphi_2}(\boldsymbol{\xi}_k)$ networks.
The time evolution is then obtained recursively. At each time step, the next extended state $\boldsymbol{\xi}_{k+1}$ is computed using~\eqref{eq:affine-extended}–\eqref{eq:beta}, where the current encoded input $\mathbf{u}_{k}^{\prime}$ enters the model in an affine manner, consistent with the structure illustrated in Fig.~\ref{fig:framework}.
To assess the prediction accuracy, the predicted latent states are compared against reference latent trajectories obtained by encoding the corresponding ground-truth temperature fields with the state encoder.
A very good agreement is observed between the predicted and true latent vectors, with an average RMSE of $2.26 \times 10^{-3}$, computed across all latent dimensions and time steps.
It should be noted that Figs.~\ref{fig:heat_recons} and \ref{fig:eval_latent_heat} correspond to two different simulations parameterized by distinct input profiles, illustrating the model's ability to generalize across different scenarios.

\subsubsection{Comparison of control-affine and baseline ROMs}
The performance of the control-affine ROM is compared to a baseline ROM with linear latent dynamics.
To ensure a fair comparison, we retain the autoencoder architecture and training algorithm from Sec.~\ref{sec:control-affine}, while constraining the latent dynamics to evolve linearly.
The reason for the choice of this linear baseline model is its relevance to our motivation.
Our framework first identifies control-affine model, and then leverage feedback linearization to obtain a (partially) linear system.
In contrast, the linear baseline model attempts to directly identify a transformation that yields linear dynamics.
The comparison between the two models is presented in Tab.~\ref{tab:comp-heat} in terms of average and standard deviation RMSEs obtained from 100 test simulations of the heat equation.
The control-affine ROM outperforms the baseline ROM in both end-to-end (i.e., the error between the predicted and ground-truth high-dimensional temperature fields, as in Fig.~\ref{fig:heat_recons}) and latent predictions (i.e., the error between the predicted latent trajectories and the encoded high-dimensional sequences, as in Fig.~\ref{fig:eval_latent_heat}).
It achieves approximately a twofold reduction in reconstruction error, with an even more pronounced improvement in the latent dynamics accuracy.
In addition to lower mean RMSE values, the control-affine model also exhibits slightly smaller standard deviations, indicating more consistent performance across different simulations.
However, the overall performance of the linear baseline ROM remains strong in this example, which can be attributed to the intrinsic linearity of the heat equation, as its dynamics can be effectively captured by a linear model in the latent space.
\begin{table}[ht]
\centering
\caption{Average RMSEs and standard deviations computed over 100 simulations for the heat equation.}
\label{tab:comp-heat}
\begin{tabular}{lll}
\hline
 RMSE & Control-affine ROM & Baseline ROM \\
\hline
End-to-end  & $7.70 \times 10^{-3} \pm 5.9\times10^{-3}$ & $1.62 \times 10^{-2} \pm 6.5\times10^{-3}$ \\
Latent  & $3.24\times 10^{-3} \pm 2.5\times10^{-3}$ & $2.09 \times 10^{-2} \pm 6.8\times10^{-3}$ \\
\hline
\end{tabular}
\end{table}

\subsubsection{Control using feedback linearization}
\begin{figure}[t]
    \centering
    \includegraphics[width=\textwidth]{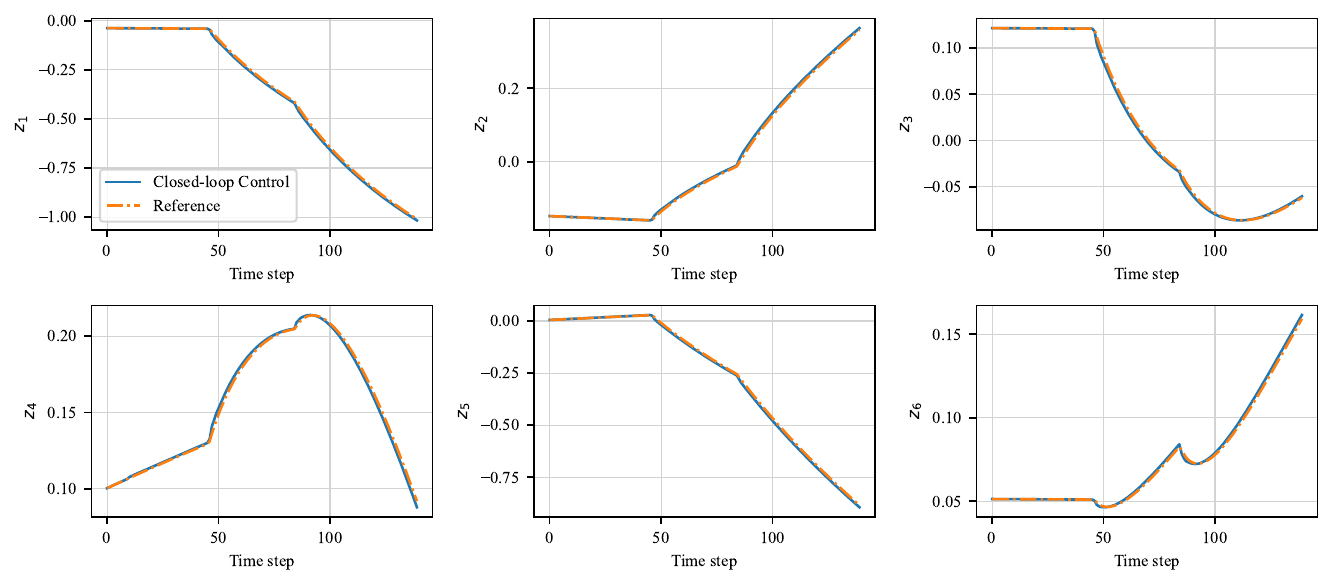}
    \caption{Latent space reference tracking using PID control in combination with feedback linearization. The reference trajectory is obtained by encoding a test simulation of the heat equation.}
    \label{fig:heat_control}
\end{figure}
\begin{figure}[t]
    \centering
    \includegraphics{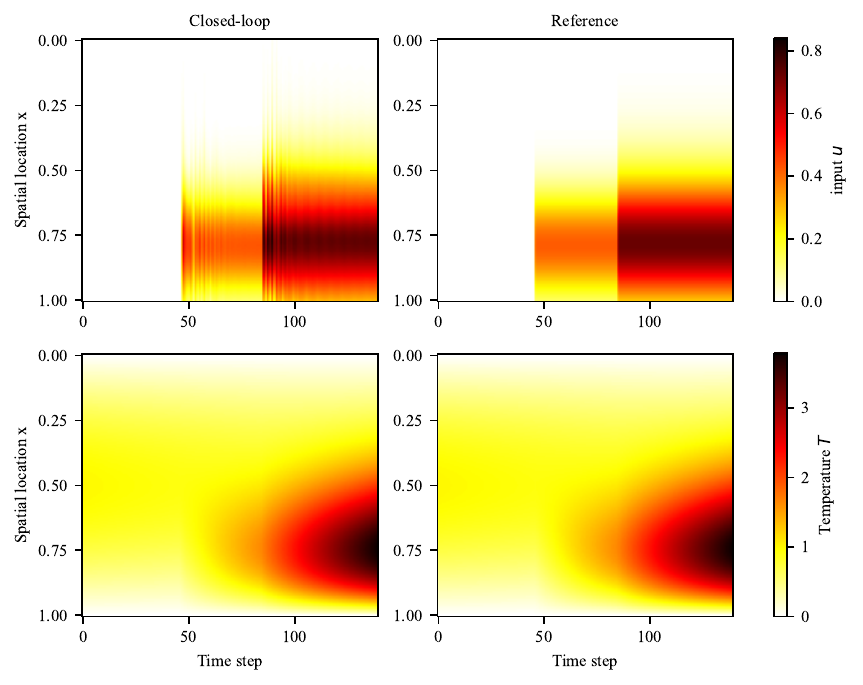}
    \caption{Comparison of decoded closed-loop control inputs and reference signals (top row), and corresponding reconstructed state trajectories versus reference simulation (bottom row).}
    \label{fig:heat_inputs}
\end{figure}
The control-affine structure of the ROM is leveraged to steer the temperature field toward a desired high-dimensional reference trajectory.
Specifically, feedback linearization is employed to compensate for the nonlinearity of the ROM, and the resulting linearized system is regulated using a PID controller, as illustrated in Fig.~\ref{fig:Blockdiagram}.
The control objective is defined within the latent space, where the reference latent vector $\mathbf{z}^{\text{ref}}_k$ is obtained by encoding the corresponding high-dimensional reference temperature profile $\mathbf{x}^{\mathrm{ref}}_k$, i.e., $\mathbf{z}^{\text{ref}}_k=\mathcal{E}_\theta(\mathbf{x}^{\mathrm{ref}}_k)$.

The closed-loop results are visualized in Fig.~\ref{fig:heat_control}, where the PID controller successfully tracks the desired latent reference trajectory.
At each time step, the PID controller produces a virtual control input $\mathbf{v}_k$, which is subsequently transformed into the latent input vector $\mathbf{u}^\prime_k$.
Notably, the matrix $\widetilde{\mathbf{B}}_{\varphi_2}(\boldsymbol{\xi}_k)$ remained full rank at all inversion steps during this experiment, despite no explicit rank-promoting term in the training loss.
The corresponding high-dimensional input profile is obtained by decoding $\mathbf{u}^\prime_k$ with the trained input decoder, i.e., $\mathbf{u}_k = \mathcal{D}^\prime_{\phi^\prime}(\mathbf{u}^\prime_k)$.
Figure \ref{fig:heat_inputs} illustrates the performance of the closed-loop controller by comparing the decoded input signal and the resulting temperature evolution with their respective references. The reference trajectory (bottom row) is designed to generate a localized hot spot near $x \approx 0.7$--$1.0$ that appears and disappears during two distinct time intervals. Physically, this corresponds to switching a localized heating element on and off to inject energy into the system. As shown in the top row, the decoded input closely matches the reference in both timing and spatial localization, as well as in magnitude, indicating that the controller successfully reproduces the actuation required to track the desired thermal state.

\subsection{Boxed ball}
\label{sec:ball}
The second example consists of a 2D simulated environment in which a  $64\time 64$ pixel camera observes a ball moving within a unit box with strong non-linear repulsive effects near the boundaries.
The system is actuated by applying external forces to the center of the ball in both $x$ and $y$ directions.
It should be noted that the same example has been treated by \cite{beintema2021non}, where a nonlinear state-space model has been identified from video data.
Here, the goal is to construct a data-driven control-affine ROM from a record of camera frames and input forces, and to design a controller that exploits the structure of the ROM using feedback linearization.
This setting reflects a common scenario in real-world applications, where models must be built solely from recorded simulations or experimental data.

We generate the training database by simulating trajectories for the movement of the ball, which is governed using the following nonlinear ODE:
\begin{subequations}\label{eq:dyn-ball}
\begin{align}
    \label{eq:dyn-ball-x}
    \Ddot{p}_x = \frac{1}{200}\left(\frac{1}{p_x^2}-\frac{1}{(1-p_x)^2}\right)-0.79\dot{p}_x + 0.25u_x, \\
    \label{eq:dyn-ball-y}
    \Ddot{p}_y = \frac{1}{200}\left(\frac{1}{p_y^2}-\frac{1}{(1-p_y)^2}\right)-0.79\dot{p}_y + 0.25u_y,
\end{align}
\end{subequations}
where $(p_x,p_y)$ denote the coordinates of the center of the ball and $(u_x,u_y)$ are the applied forces, constrained within the interval $[-1,1]$.
Each simulation camera frame is generated synthetically by calculating the pixel intensity $I(X,Y)$ at coordinates $(X,Y)$ on a $64\times 64$ cartesian grid according to:
\begin{equation}
    \label{eq:pix}
    I(X,Y) = \text{max}(0,1-\frac{(X-p_x)^2+(Y-p_y)^2}{0.25^2})+v,
\end{equation}
where $v$ is an additive noise term.
This expression produces a circular intensity profile centered at $(p_x,p_y)$, with a radius of approximately $0.25$.
The intensity decreases quadratically with the distance from the center and is truncated at zero to ensure non-negative pixel values.
As a result, the ball appears in the image as a smooth, bright disk whose position varies continuously according to \eqref{eq:dyn-ball}, while the noise term $v$ introduces mild measurement perturbations that mimic realistic sensing conditions.

\subsubsection{Data-driven control-affine ROM}
The training dataset is generated from a trajectory obtained by integrating the dynamics in \eqref{eq:dyn-ball} starting from an initial condition corresponding to $p_x(0)=p_y(0)=0.5$ with zero initial velocity.
The system is simulated for 5000 iterations with time step $\Delta t=0.3$ using the Runge-Kutta scheme.
We use a zero-order hold condition for the inputs, which are sampled independently at each step from a uniform distribution $u_x,u_y \sim \mathcal{U}(-1,1)$.
The high-dimensional frames are constructed according to \eqref{eq:pix}, with an additive Gaussian noise $v \sim \mathcal{N}(0,\sigma_y^2)$ applied independently to each pixel, where $\sigma_y=0.204$.
We emphasize that the inclusion of noise serves to improve robustness of the learned ROM and it reflects the inherent uncertainty in real-world sensing. 
Similarly, we generate 1000 noisy validation frames and 4000 test frames for evaluation.
We normalize the training, validation and testing frames between 0 and 1 to improve numerical stability and training efficiency of the ROM.

The control-affine ROM framework presented in Sec.~\ref{sec:control-affine} is formulated for general high-dimensional states $\mathbf{x}_k \in \mathcal{X}$.
In this example, the state corresponds to spatially structured data and is therefore represented as a two-dimensional field $\mathbf{x}_k \in \mathbb{R}^{N_X \times N_Y}$.
Consistent with the general formulation, such structured states can be interpreted either through their vectorized representation or, more naturally, by preserving their spatial organization.
In the latter case, the encoder and decoder mappings are parameterized using convolutional neural networks, which exploit the underlying spatial correlations.
The architecture of the encoder employed in this example is detailed in Tab.~\ref{tab:encoder-ball}.
It processes a single-channel $64 \times 64$ image through a sequence of 2D convolutional layers with ReLU activation functions, followed by a flattening operation and fully connected layers.
This design progressively extracts and compresses spatial features from the high-dimensional state $\mathbf{x}_k \in \mathbb{R}^{N_X \times N_Y}$ into a latent representation $\mathbf{z}_k = \mathcal{E}_\theta(\mathbf{x}_k) \in \mathbb{R}^r$.
The latent space is constrained between 0 and 1 using a final sigmoid activation.
We note the decoder adopts a symmetric architecture, as it performs the counterpart inverse operation.
In this example, the latent dimension is set to $r=m=2$, yielding to a square input matrix and enabling full-state latent feedback linearization in the latent space (cf. Proposition~\ref{prop:cond-FL-extended}).
All remaining hyperparemeters, including the number of layers, kernel size, activation functions, have been tuned manually to obtain the most accurate ROM evaluated on the validation dataset.
For the control-affine latent model, the architecture of the drift and input networks consist each of 3 fully connected hidden layers composed of 256 neurons with ReLU activations, followed with a linear output layer.
The drift network outputs $r=2$ neurons corresponding to the dimension of the latent vector, while the input network outputs $r \cdot m=4$ neurons, which are subsequently reshaped into an $r \times m=2 \times 2$ control input matrix.
It should be noted that this example does not employ an input autoencoder, as the input vector is already low-dimensional.
\begin{table}[ht]
\centering
\caption{Architecture of the encoder network for the boxed ball example. The input is a single-channel frame of size $64 \times 64$. All the convolutional layers have a kernel of size $3\times 3$ and a stride of 2.}
\label{tab:encoder-ball}
\begin{tabular}{ll}
\hline
Layer & Outputs \\
\hline
input & $1 \times 64 \times 64$ \\
Conv2D(3,4), ReLU & $4 \times 32 \times 32$ \\
Conv2D(3,8), ReLU & $8 \times 16 \times16$ \\
Conv2D(3,16), ReLU & $16 \times 8 \times 8$ \\
Conv2D(3,32), ReLU & $32 \times 4 \times 4$ \\
Flatten & 512 \\
Dense(128), ReLU & 128 \\
Dense($r$), Sigmoid & $r$ \\
\hline
\end{tabular}
\end{table}

The ROM is trained following Algorithm~\ref{alg:rom-training} using ADAM optimizer with initial learning rate $10^{-3}$, a batch size of 64, and a step learning rate decay.
In the first stage, we train the autoencoder alone for 20 epochs.
Subsequently, the full ROM is trained for 500 epochs using single-step rollouts, i.e., $M=1$, and history sequences of length $H=4$.
The coefficients of the multi-objective loss function in \eqref{eq:mult-obj-loss} are the same ones used in the previous example, i.e., $\gamma_1=\gamma_2=1$ and $\gamma_3=0.3$.
During training, model checkpoints are employed to retain the model parameters with the lowest validation loss.
For evaluation, we compare the dynamics predicted by the trained ROM against ground truth simulations, using randomly chosen sequences of initial conditions and input signal drawn from the test data.
The evaluation is performed in a recursive manner over 500 rollout steps.
The results depicted in Fig.~\ref{fig:eval_test} show very good agreement between the predicted and ground truth frames.
The error plots shown therein reveal that for most predictions, the position of the ball is correctly captured and the error primarily consists of eliminating noise.
The only exception is the prediction at time step $k=284$, where the predicted position of the ball is slightly shifted.
The overall RMSE, computed over all spatial points and time steps in the simulation horizon, is $3.32 \times 10^{-2}$.
\begin{figure}[t]
    \centering
    \includegraphics[width=\textwidth]{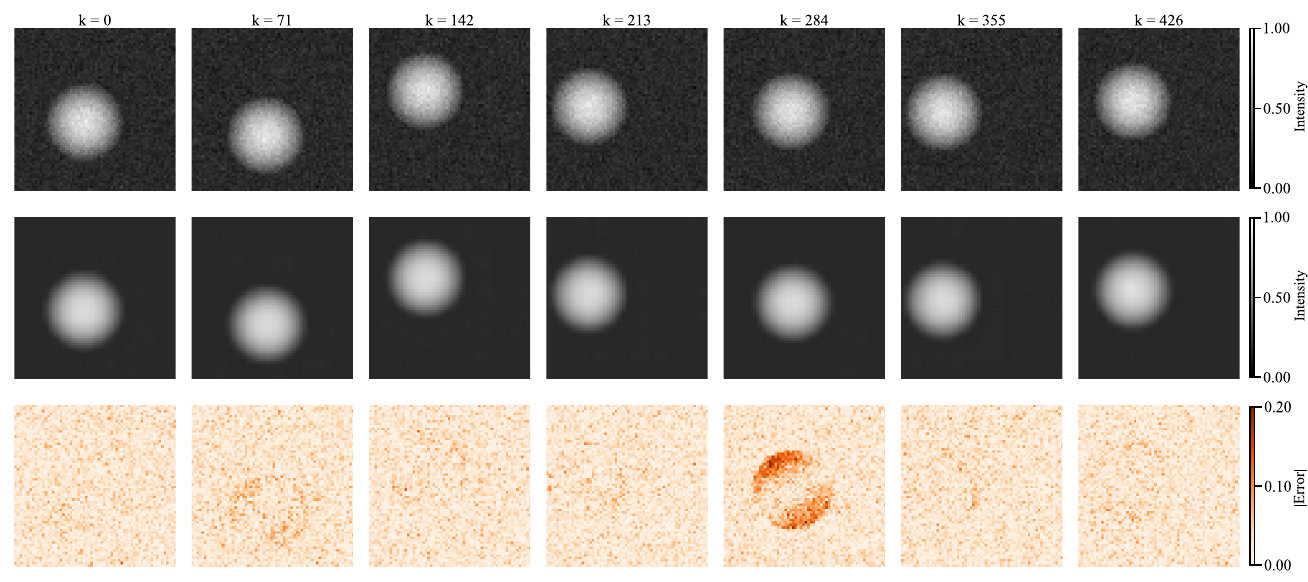}
    \caption{Comparison of (top) ground-truth frames, (middle) ROM-predicted outputs, and (bottom) absolute error at different time steps, obtained from randomly chosen initial conditions and input signal for the ball example.}
    \label{fig:eval_test}
\end{figure}

In addition to the end-to-end evaluation, we further assess the latent-space dynamics governed by \eqref{eq:affine-extended}.
Specifically, an initial sequence of camera frames is encoded into its corresponding latent representation.
This latent sequence is then combined with the input signal to construct the extended state vector $\boldsymbol{\xi}_k$.
Starting from this initialization, the model recursively predicts future latent states according to \eqref{eq:affine-extended}, where each predicted latent vector, together with the subsequent input, is used to update $\boldsymbol{\xi}_k$ and advance the dynamics.
Figure~\ref{fig:eval_latent_ball} presents the results of this evaluation, comparing the predicted latent trajectories with the ground-truth latent vectors.
The ground-truth vectors are obtained by encoding the corresponding frames of a randomly selected sequence from the test data.
Across 700 rollouts, the control-affine model demonstrates high accuracy in predicting latent dynamics, achieving an average RMSE of $1.82\times10^{-2}$ computed across all latent dimensions and time steps. Notably, this performance is attained without relying on the decoder to correct latent predictions. Furthermore, although the ROM was trained using only single-step rollouts, it successfully maintains high prediction accuracy over extended horizons.
\begin{figure}[t]
    \centering
    \includegraphics{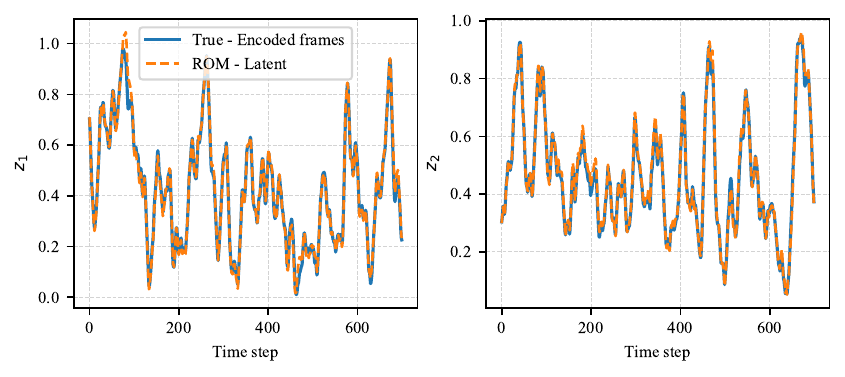}
    \caption{Comparison of the predicted latent dynamics using the control-affine model and the ground-truth encoded frames, obtained from randomly chosen initial conditions and input signal for the ball example.}
    \label{fig:eval_latent_ball}
\end{figure}

In Fig.~\ref{fig:eval_latent_autonomous}, we evaluate the autonomous dynamics by setting the input signal to zero. 
Ground-truth frames are generated by integrating the system in \eqref{eq:dyn-ball} from randomly chosen initial positions and velocities.
We compare the encoded ground-truth trajectories to those predicted by the control-affine latent model when only the drift network is active.
The results show good agreement between predicted and reference trajectories, with minor discrepancy near the steady state.
To further evaluate the autonomous dynamics, we identify the latent equilibrium of the trained drift network and compare it to the encoded physical equilibrium $\mathbf{z}^*=\mathcal{E}_\theta(\mathbf{x}^*)$, where $\mathbf{x}^*$ denotes the equilibrium frame identified from \eqref{eq:dyn-ball}, corresponding to the ball being at the center of the box, i.e., $(p_x^*,p_y^*)=(0.5,0.5)$.
The equilibrium state of the model is obtained by solving a fixed-point problem in the latent space.
Specifically, we seek a state $\Tilde{\mathbf{z}}^*$ that satisfies $\Tilde{\mathbf{z}}^* = \mathbf{a}_{\varphi_1}(\Tilde{\mathbf{z}}^*)$, where $\mathbf{a}_{\varphi_1}$ represents the trained drift network.
To this end, the latent vector $\Tilde{\mathbf{z}}^*$ is treated as an optimization variable and iteratively updated using the ADAM optimizer to satisfy the fixed-point condition.
The resulting absolute error between the identified and the true equilibrium is $[0.024, 0.002]$, confirming close agreement.
While the training framework presented in Sec.~\ref{sec:control-affine} could include an additional loss term corresponding to the equilibrium, the results show that the model closely identifies the equilibrium even without exposure to autonomous trajectories during training.
\begin{figure}[t]
    \centering
    \includegraphics{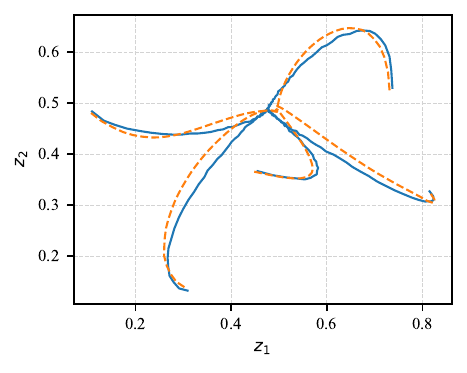}
    \caption{Comparison of the (dashed lines) predicted autonomous latent trajectories using the control-affine model and the (solid lines) ground-truth encoded frames, obtained from randomly chosen initial conditions for the ball example.}
    \label{fig:eval_latent_autonomous}
\end{figure}

\subsubsection{Comparison of control-affine and baseline ROMs}
Similar to the previous example, the performance of the control-affine ROM is compared against a baseline model with linear latent dynamics; however, in contrast to the previous case, both forced and autonomous simulations are considered here.
The results reported in Tab.~\ref{tab:CAvsLin} in terms of RMSEs are averages and standard deviations obtained from 100 simulations with randomly chosen initial conditions and input signals.
Our experiments show that the control-affine ROM outperforms the baseline model in the forced setting, achieving roughly a factor of two lower end-to-end RMSE and more than a threefold reduction in latent-space error.
Furthermore, the control-affine model also exhibits noticeably smaller standard deviations, indicating improved robustness with respect to random initialization.
In the autonomous setting, the end-to-end errors of both models are nearly identical, suggesting that the decoder of the baseline model is able to partially compensate for inaccuracies in the linear latent dynamics.
However, this effect is not reflected in the latent space, where the control-affine ROM achieves significantly lower errors and reduced variance.
It should be noted that both the baseline and control-affine ROMs were trained using the same loss function that is based on single step predictions ($M=1$), while their performance was evaluated through recursive simulations.
We believe that training the baseline ROM with $M-$step rollout loss akin to \cite{lusch2018deep} would improve its performance.
However, in order to enable a direct comparison with our model, we maintain the same training setup.
\begin{table}[ht]
\centering
\caption{Average RMSEs and standard deviations computed over 100 random initializations for both forced and autonomous simulations for the ball example.}
\label{tab:CAvsLin}
\begin{tabular}{lll}
\hline
 RMSE & Control-affine ROM & Baseline ROM \\
\hline
Forces End-to-end  & $3.47 \times 10^{-2} \pm 1.4\times10^{-3}$ & $6.03 \times 10^{-2} \pm 6.4\times10^{-3}$ \\
Forced Latent  & $1.57\times 10^{-2} \pm 2.8\times10^{-3}$ & $5.62 \times 10^{-2} \pm 1.1\times10^{-2}$ \\
Autonomous End-to-End  & $3.31\times 10^{-2} \pm 8.6\times10^{-4}$ & $3.33 \times 10^{-2} \pm 1.6\times10^{-3}$ \\
Autonomous Latent  & $1.25\times 10^{-2} \pm 1.7\times10^{-3}$ & $1.84 \times 10^{-2} \pm 4.4\times10^{-3}$ \\
\hline
\end{tabular}
\end{table}

\subsubsection{Control using feedback linearization}
We control the trajectory of the ball by leveraging the control-affine structure of the ROM as presented in Sec.~\ref{sec:FL}.
In this example, we demonstrate reference tracking control for a circular ball trajectory within the confined environment. 
The controller acts on the feedback-linearized latent dynamics through a proportional law, and the resulting virtual input is mapped back to the physical control input through the learned state-dependent transformation.
To respect the actuator limits, the applied control input is clamped to the interval $[-1,1]$.
This makes the closed-loop behavior exactly linear only during the unsaturated intervals, whereas the saturated portions introduce a piecewise nonlinear behavior.
This is consistent with the constrained-feedback-linearization literature, which emphasizes that simple input bounds become state dependent after the nonlinear transformation introduced by feedback linearization.
For the reported simulation, clamping was active for approximately $30\%$ of the time steps, which quantifies the departure from the ideal linear regime.
Notably, the matrix $\widetilde{\mathbf{B}}_{\varphi_2}(\boldsymbol{\xi}_k)$ remained full rank at all inversion steps in this experiment, despite the absence of an explicit rank-promoting term in the training loss.
As shown in Fig.~\ref{fig:ref-tracking}, the controller successfully drives the latent states to track the encoded reference trajectory with minimal error.

Although not applied here, the broader feedback-linearization literature provides methods for handling input constraints.
For example, \cite{deng2009input} study input-constraint handling in an MPC/feedback-linearization setting by transforming the feasible input region, while \cite{tiriolo2024feedback} derive state-dependent input constraints for a feedback-linearized model and construct a time-invariant inner approximation to preserve feasibility.
These works suggest that constraint-aware control of control-affine systems can often be reformulated in a linear or linearly constrained form in suitable coordinates.
In our setting, however, the drift and input maps are represented by neural networks, so adapting these methods to the proposed ROM architecture is not immediate and is therefore left for future work.
Nevertheless, the existence of such results further motivates learning reduced-order models with an explicit control-affine structure.


\begin{figure}[t]
    \centering
    \includegraphics{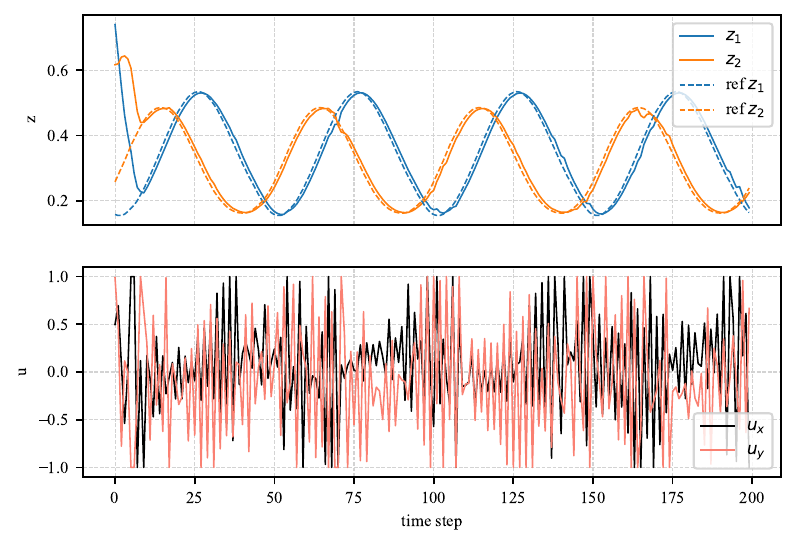}
    \caption{Latent space reference tracking using P control in combination with feedback linearization (upper). A high-dimensional circular ball trajectory is encoded to a latent reference (dashed lines). The controller drives the latent states to track this reference while generating the necessary control signals (lower).}
    \label{fig:ref-tracking}
\end{figure}

\section{Conclusions}
\label{sec:conclusions}
In this work, a data-driven framework for learning control-affine ROMs via AEs has been presented.
The framework identifies a low-dimensional manifold in which the system dynamics are well-approximated by a control-affine latent model.
This is achieved by formulating a multi-objective loss function for end-to-end training, i.e., simultaneously training the autoencoder and the dynamical model.
To improve prediction accuracy, we formulate the dynamical model to operate on latent sequences while preserving the control-affine structure.
In addition, the framework accommodates high-dimensional inputs—such as those arising from spatially distributed controls—through the optional integration of an input autoencoder, which is trained jointly with the other components.

Our work builds upon related work that employs AEs to learn latent dynamics with predefined structure, e.g., linear or sparsely nonlinear dynamics. 
We demonstrate, however, that the control-affine latent dynamics yields more accurate results compared to the linear latent dynamics in our numerical examples.
In addition, the control-affine structure is well-studied in control-theory with decades of theoretical and practical insights that are less readily available for general nonlinear latent representations.
Specifically, we motivate here the adoption of control-affine dynamics by demonstrating their use in feedback linearization to track the latent states to a desired reference trajectory. 

To this end, we showed that full-state feedback linearization is achievable by setting the latent dimension $r$, which is a hyperparameter of the ROM, less than or equal to the number of control inputs $m$.
Yet, this requirement may be restrictive in practice, especially for very high-dimensional systems with only a few actuators.
In such settings, a control-affine ROM with a larger latent dimension can still be learned accurately, but only partial latent feedback linearization becomes feasible.
This case poses several challenges: the control objective cannot be directly encoded in the latent space, and identifying which latent coordinates correspond to actuated directions is difficult without additional regularization. 
In the future, we aim to address these limitations by explicitly decomposing the latent dynamics into controllable, feedback-linearizable subspace and an internal, unactuated subspace.
Specifically, we aim to expand the framework to promote controllability of the learned control-affine sub-latent dynamics and automatically identify the latent coordinates most relevant for control.

Finally, we provide an open-source Python package, DeepE2EROM, which implements the end-to-end training framework developed in this work.
The package offers a high-level interface for constructing reduced-order models directly from data, enabling seamless integration of state and input autoencoders with a variety of latent dynamical models.
The control-affine latent dynamics structure introduced in this paper is fully supported, together with linear dynamics models, LSTM-based nonlinear dynamics, and user-defined custom architectures.
DeepE2EROM thus serves as a flexible research tool for developing end-to-end ROMs across a wide range of applications.

\section*{Data accessibility}
The developed framework and the data used are made available at: https://github.com/mjalled/DeepE2EROM.

\section*{Acknowledgment}
Funded by the Deutsche Forschungsgemeinschaft (DFG, German Research Foundation) - Project-ID 422037413 - TRR 287.

\bibliographystyle{plainnat}
\bibliography{references}

\end{document}